\documentclass[a4paper,reqno]{amsart}
\pdfoutput=1
\usepackage[allcolors=blue,colorlinks]{hyperref}
\usepackage{amssymb,amsfonts}
\usepackage{mathrsfs}
\usepackage[numbers]{natbib}
\usepackage[all,cmtip]{xy}
\usepackage{color}
\usepackage[utf8]{inputenc}
\usepackage{fullpage}

\newcommand{\LF}[2]{\langle #1,#2\rangle}
\newcommand{\DLF}[2]{\langle\!\langle #1,#2\rangle\!\rangle}

\newcommand{\om}{\omega}

\newcommand{\ra}{\rightarrow}
\newcommand{\lra}{\longrightarrow}
\newcommand{\ZZ}{\mathbb{Z}}

\newcommand{\QQ}{\mathbb{Q}}
\newcommand{\RR}{\mathbb{R}}
\newcommand{\Rr}{\mathcal{R}}
\newcommand{\Ee}{\mathcal{E}}

\newcommand{\NN}{\mathbb{N}}

\newcommand{\AAA}{\mathbb{A}}
\newcommand{\XX}{\mathbb{X}}
\newcommand{\Cc}{\mathcal{C}}
\newcommand{\Qq}{\mathcal{Q}}
\newcommand{\Dd}{\mathcal{D}}

\newcommand{\Aa}{\mathcal{A}}
\newcommand{\Xx}{\mathcal{X}}

\newcommand{\Ll}{\mathcal{L}}
\newcommand{\Hh}{\mathcal{H}}
\newcommand{\QHh}{\vec{\mathcal{H}}}
\newcommand{\Tt}{\mathcal{T}}

\newcommand{\Ff}{\mathcal{F}}
\newcommand{\Gg}{\mathcal{G}}
\newcommand{\Kk}{\mathcal{K}}
\newcommand{\Ss}{\mathcal{S}}
\newcommand{\vSs}{\mathscr{S}}
\newcommand{\Uu}{\mathcal{U}}
\newcommand{\Vv}{\mathcal{V}}
\newcommand{\Ww}{\mathcal{W}}
\newcommand{\Bb}{\mathcal{B}}
\newcommand{\Mm}{\mathcal{M}}
\newcommand{\bp}{\mathbf{p}}
\newcommand{\ovp}{\bar{p}}
\newcommand{\bt}{\mathbf{t}}
\newcommand{\bq}{\mathbf{q}}

\newcommand{\bL}{\mathbf{L}}
\newcommand{\bW}{\mathbf{W}}
\newcommand{\spitz}[1]{\langle #1\rangle}
\newcommand{\rperp}[1]{#1^{\perp}}
\newcommand{\lperp}[1]{{}^{\perp}#1}
\newcommand{\rperpo}[1]{#1^{\perp_0}}
\newcommand{\lperpo}[1]{{}^{\perp_0}#1}
\newcommand{\rperpe}[1]{#1^{\perp_1}}
\newcommand{\lperpe}[1]{{}^{\perp_1}#1}

\DeclareMathOperator{\Knull}{K_0}
\DeclareMathOperator{\D}{D}
\DeclareMathOperator{\Coker}{Coker}
\DeclareMathOperator{\Ker}{Ker}

\DeclareMathOperator{\vect}{vect}
\DeclareMathOperator{\coh}{coh}
\DeclareMathOperator{\Qcoh}{Qcoh}
\DeclareMathOperator{\Inj}{Inj}

\DeclareMathOperator{\Mod}{Mod}
\renewcommand{\mod}{\operatorname{mod}}
\DeclareMathOperator{\End}{End}
\DeclareMathOperator{\Hom}{Hom}

\DeclareMathOperator{\Ext}{Ext}
\DeclareMathOperator{\Gen}{Gen}
\DeclareMathOperator{\gen}{gen}
\DeclareMathOperator{\Cogen}{Cogen}
\DeclareMathOperator{\add}{add}
\DeclareMathOperator{\Add}{Add}
\DeclareMathOperator{\Prod}{Prod}

\DeclareMathOperator{\rad}{rad}

\DeclareMathOperator{\op}{op}

\DeclareMathOperator{\her}{her}

\DeclareMathOperator{\Copres}{Copres}

\DeclareMathOperator{\fp}{fp}

\DeclareMathOperator{\id}{id}

\DeclareMathOperator{\rk}{rk}
\DeclareMathOperator{\matring}{M}

\DeclareMathOperator{\Ab}{Ab}
\DeclareMathOperator{\IE}{E}
\DeclareMathOperator{\PE}{PE}
\DeclareMathOperator{\Lex}{Lex}
\DeclareMathOperator{\y}{y}

\newcommand{\Der}{\mathcal{D}}
\newcommand{\Derived}[1]{\Der(#1)}
\newcommand{\bDerived}[1]{\Der^b(#1)}

\newtheorem*{theorem-1}{Theorem~1}
\newtheorem*{theorem-2}{Theorem~2}
\newtheorem*{theorem-3}{Theorem~3}
\newtheorem*{theorem-4}{Theorem~4}
\newtheorem*{theorem-5}{Theorem~5}
\newtheorem*{theorem*}{Theorem}
\newtheorem*{corollary*}{Corollary}

\newtheorem{proposition}{Proposition}[section]
\newtheorem{theorem}[proposition]{Theorem}
\newtheorem{corollary}[proposition]{Corollary}

\newtheorem{lemma}[proposition]{Lemma}

\theoremstyle{definition}
\newtheorem{definition}[proposition]{Definition}

\newtheorem{remark}[proposition]{Remark}
\newtheorem{numb}[proposition]{\!\!}
\newtheorem{question}[proposition]{Question}
\newtheorem{example}[proposition]{Example}

\numberwithin{equation}{section}
\numberwithin{figure}{section}
\begin{document}

\title[Cotilting sheaves over weighted curves]{Cotilting
  sheaves \\ over weighted noncommutative\\
  regular projective curves}
\author[D. Kussin]{Dirk Kussin} \address{Technische Universit\"at Berlin \\
  Institut f\"ur Mathematik \\ Stra{\ss}e des 17.~Juni 136 \\ 10623
  Berlin, Germany} \email{dirk@math.uni-paderborn.de}
\author[R. Laking]{Rosanna Laking} \address{Universit\`a degli Studi di Verona\\
  Strada Le Grazie 15 - Ca' Vignal 2\\
  I - 37134 Verona\\
  Italy} \email{rosanna.laking@univr.it}
\begin{abstract}
  We consider the category $\Qcoh\XX$ of quasicoherent sheaves where
  $\XX$ is a weighted noncommutative regular projective curve over a
  field $k$. This category is a hereditary, locally noetherian
  Grothendieck category. We classify all indecomposable pure-injective
  sheaves and all cotilting sheaves of slope $\infty$. In the cases of
  nonnegative orbifold Euler characteristic this leads to a
  classification of pure-injective indecomposable sheaves and a
  description of all large cotilting sheaves in $\Qcoh\XX$.
\end{abstract}
\maketitle
\tableofcontents
\section{Introduction}
The study of large cotilting objects originates in the context of the
representation theory of associative rings, where it amounts to the
study of (tilting) derived equivalences between the module category
and Grothendieck categories (\cite{stovicek:2014}).  A generalisation
of cotilting to the setting of Grothendieck categories was provided in
\cite{buan:krause:2003}, and investigated in greater depth in
\cite{coupek:stovicek:2019}. As cotilting objects are automatically
pure-injective (unlike the dual notion of a tilting object), the
classifications of cotilting objects and of indecomposable
pure-injective objects are strongly related to each other. In this
paper we consider these classification problems for a certain class of
Grothendieck categories that are not module categories: the categories
$\Qcoh{\XX}$ of quasicoherent sheaves over weighted noncommutative
regular projective curves over a field $k$. We emphasize that each
smooth projective curve is included in this setting as a special case.

Each such category $\Qcoh{\XX}$ is determined by its full subcategory
$\coh{\XX}$ of finitely presented objects.  The category $\coh{\XX}$
is, by definition, a $k$-linear abelian category that shares important
characteristics with classical categories of coherent sheaves over
(commutative) projective curves. In fact, the categories $\coh{\XX}$
have been axiomatised (\cite{lenzing:reiten:2006}) and subsequently
studied by several authors (for example,
\cite{kussin:2016,angeleri:kussin:2017}). In particular, the category
$\coh{\XX}$ is a small hereditary abelian category in which every
object is noetherian.

The structure of the category $\Qcoh{\XX}$ is less well-understood
than $\coh{\XX}$ and is likely to be beyond any hope of classification
or description as a whole.  In this article, we systematically study
the full subcategory of pure-injective sheaves in $\Qcoh{\XX}$, in the
sense of \cite{crawley-boevey:1994, Krause2:1998}. This subcategory
properly contains the subcategory $\coh{\XX}$ of coherent sheaves and,
moreover, it constitutes a tractible subcategory of $\Qcoh{\XX}$, due
to the fact that we may make use of the pure-exact structure.

For arbitrary $\XX$ we are able to give the following description of
the indecomposable pure-injective sheaves $E$ of slope $\infty$, that
is, those which satisfy additionally $\Hom(E,\vect{\XX})=0$.
\begin{theorem*}[\ref{thm:ind-pure-inj-slope-infty}]\label{thm:intro-indec-pure-injectives-slope-infty}
  Let $\XX$ be a weighted noncommutative regular projective curve over
  a field $k$. The following is a complete list of indecomposable
  pure-injective objects in $\Qcoh\XX$ of slope $\infty$.
  \begin{enumerate}
  \item The indecomposable sheaves of finite length.
  \item The sheaf $\Kk$ of rational functions, the Pr\"ufer and the
    adic sheaves.
  \end{enumerate}
  Moreover, each pure-injective sheaf of slope $\infty$ is discrete
  and thus uniquely determined by its indecomposable summands.
\end{theorem*}
If we assume that $\XX$ is of tame representation type (which means
that the orbifold Euler characteristic of $\XX$ is nonnegative), then
we can extend this classification to the sheaves of rational and
infinite slope.  In the case of positive Euler characteristic, we
describe all the indecomposable pure-injective sheaves in
$\Qcoh{\XX}$. In particular, we show that, when the orbifold Euler
characteristic of $\XX$ is nonnegative, the form of the indecomposable
pure-injective sheaves is analogous to the case of modules over
concealed canonical algebras (\cite{angeleri:kussin:2017b}). We recall
that in case of orbifold Euler characteristic zero each indecomposable
object has a slope, which is a real number or infinite,
by~\cite{reiten:ringel:2006,angeleri:kussin:2017}.
\begin{theorem*}[\ref{thm:domestic-main-result} and \ref{thm:nondomestic-main-result}]\label{thm:ind-pure-inj}
  Let $\chi'_{orb}(\XX)$ denote the orbifold Euler characteristic of
  $\XX$. Then the following statements hold.
  \begin{enumerate}
  \item If $\chi'_{orb}(\XX) > 0$ (i.e.~if $\XX$ is a domestic curve),
    then each indecomposable pure-injective sheaf in $\Qcoh{\XX}$
    either has slope $\infty$, and thus is as in the preceding
    theorem, or is a vector bundle.\smallskip
  \item If $\chi'_{orb}(\XX) = 0$ (i.e.~if $\XX$ is a tubular or an
    elliptic curve), then the following is a complete list of
    indecomposable pure-injective sheaves of rational or infinite
    slope $w$ in $\Qcoh{\XX}$.\smallskip
\begin{enumerate}
  \item The indecomposable coherent sheaves.
  \item The generic, the Pr\"ufer and the adic sheaves of slope $w$.
    \end{enumerate}
\end{enumerate}
\end{theorem*}
We also classify the cotilting sheaves in $\Qcoh{\XX}$, which allows
us to determine the existence of pure-injective sheaves of irrational
slope. For arbitrary $\XX$, we have the following parametrisation of
the cotilting sheaves in $\Qcoh{\XX}$ of slope $\infty$; for the
complete statement we refer to
Theorem~\ref{thm:large-cotilting-sheaves-in-general}. Note that branch
sheaves are certain rigid coherent sheaves contained in
non-homogeneous tubes and are defined in Section
\ref{sec:slope-infinity}.
\begin{theorem*}[\ref{thm:large-cotilting-sheaves-in-general}]
  Let $\XX$ be a weighted noncommutative regular projective curve over
  a field $k$. The cotilting sheaves $C$ in $\Qcoh\XX$ of slope
  $\infty$ are parametrized by pairs $(B,V)$ where $V$ is a subset of
  $\XX$ and $B$ a branch sheaf. 
\end{theorem*}
In the theorem, the cotilting module $C$ is uniquely determined by its
torsion part, which is given as a direct sum of $B$ and a coproduct of
Pr\"ufer sheaves concentrated in $V$; the set of the indecomposable
summands of the torsionfree part is then given by certain
``complementing'' adic sheaves concentrated in $\XX\setminus V$ (and
$\Kk$, if $V=\emptyset$).

In the cases of nonnegative orbifold Euler characteristic we show that
every large (=non-coherent) cotilting sheaf $C$ in $\Qcoh{\XX}$ has a
well-defined slope $w$ (see Theorem~\ref{thm:domestic-main-result} and
Theorem \ref{thm:every-large-ts-slope}) and, moreover, the equivalence
class of $C$ is completely determined by a set of indecomposable
pure-injective sheaves (see Proposition \ref{prop:minimal}). We have
the following parametrisation of the large cotilting sheaves in
$\Qcoh{\XX}$.  Note that branch sheaves of rational slope are defined
in Section \ref{sec:Euler-zero}.
\begin{theorem*}[\ref{thm:domestic-main-result} and \ref{thm:Euler-zero-cotilting}]\label{thm:cotilt}
  Under the assumptions of first theorem above, the
  following statements hold.
  \begin{enumerate}
\item If $\chi'_{orb}(\XX) > 0$, then all large cotilting sheaves in
  $\Qcoh{\XX}$ are of slope $\infty$.\smallskip
\item If $\chi'_{orb}(\XX) = 0$, then all large cotilting sheaves in
  $\Qcoh{\XX}$ have a well-defined slope $w$ and are parametrised as
  follows.\smallskip
  \begin{enumerate}
  \item If $w$ is rational or $\infty$, then the large cotilting
    sheaves of slope $w$ are parametrised (up to equivalence) by pairs
    $(B_w, V_w)$ where $B_w$ is a branch sheaf of slope $w$ and
    $V_w \subseteq \XX_w$.
  \item If $w$ is irrational, then there is a unique large cotilting
    sheaf $\bW_w$ of slope $w$ (up to equivalence).
\end{enumerate}
\end{enumerate}
\end{theorem*}
In the case where $w$ is rational or $\infty$, we provide an explicit
description of the minimal set of indecomposable direct summands of
$C$ (see Theorem \ref{thm:Euler-zero-cotilting}) in terms of the
classification given in Theorem \ref{thm:ind-pure-inj-slope-infty}.
Moreover, we describe the pure-injective sheaves of irrational slope
in terms of the large cotilting sheaves given in Theorem
\ref{thm:exist-Ww}.
\begin{corollary*}[\ref{cor:pure-injectives-of-irrational-slope}]
  If $\chi'_{orb}(\XX) = 0$ and $w$ is irrational, then $\Prod(\bW_w)$
  is the set of the pure-injective sheaves of slope $w$.
\end{corollary*}
The form of the indecomposable pure-injective sheaves of irrational
slope is not known but we show that there is a direct connection
between the indecomposable direct summands of $\bW_w$ and the simple
objects in the heart $\Gg_w$ of the HRS-tilted t-structure (see
Proposition \ref{prop:minimal}).  This perspective therefore provides
an interesting future strategy for investigating the indecomposable
pure-injective sheaves of irrational slope. In particular, in relation
to the recent description of some simple objects in $\Gg_w$ in the
case where $\XX$ is of tubular type (see \cite{rapa:2019} and a
forthcoming preprint by A.~Rapa and
J.~\v{S}\v{t}ov\'{\i}\v{c}ek). Therefore we will exhibit and prove
some basic properties of the categories $\Gg_w$ in the final section.

In some sense the results presented here are ``dual" to the
description of large tilting sheaves of finite type given by the first
named author and L.~Angeleri H\"ugel (\cite{angeleri:kussin:2017}),
however there is no obvious concrete duality witnessing this
intuition.  In the absence of an explicit duality, we observe some
connection between large tilting sheaves of finite type and large
cotilting sheaves in Theorem \ref{thm:tilting-cotilting-duality},
Proposition~\ref{prop:T-C-same-torsion} and Lemma
\ref{lem:cotilt-tilt}.

We end this introduction with a summary of the structure of the paper.
In Section \ref{sec:purity} we introduce the main set of techniques we
use to establish our results: the theory of purity in locally finitely
presented Grothendieck categories and the theory of purity in
compactly generated triangulated categories.  We also prove some
preliminary results in this setting.  Next, in Section
\ref{sec:cotilting}, we introduce the definition of a cotilting object
in a Grothendieck category.  We summarise the connections between
properties of cotilting objects and the injective cogenerators in
HRS-tilted categories. In Section \ref{sec:sheaves-and-modules} we
introduce the categories of quasicoherent sheaves over weighted
noncommutative regular projective curves. In Section
\ref{sec:pure-inj-inf-slope} we classify the indecomposable
pure-injective sheaves of slope $\infty$ and in Section
\ref{sec:slope-infinity} we classify the large cotilting sheaves of
slope $\infty$; as mentioned, this is done for all orbifold Euler
characteristics. In the final sections we extend these classifications
of sheaves of slope infinity to include all slopes in the domestic and
the tubular/elliptic cases, respectively. We then study the
above-mentioned categories $\Gg_w$ in the case where $w$ is
irrational.


\subsection*{Notation} Let $\Xx$ be a class of
objects in a Grothendieck category $\Aa$. We will use the following
notation for orthogonal classes:
\begin{gather*}
  \rperpo{\Xx}=\{F\in\Aa\mid\Hom(\Xx,F)=0\},\quad
  \rperpe{\Xx}=\{F\in\Aa\mid\Ext^1(\Xx,F)=0\},\\
  \lperpo{\Xx}=\{F\in\Aa\mid\Hom(F,\Xx)=0\},\quad
  \lperpe{\Xx}=\{F\in\Aa\mid\Ext^1(F,\Xx)=0\},\\
    \rperp{\Xx}=\rperpo{\Xx}\cap\rperpe{\Xx},
    \quad \lperp{\Xx}=\lperpo{\Xx}\cap\lperpe{\Xx}.
\end{gather*}
By $\Add(\Xx)$ (resp.\ $\add(\Xx)$) we denote the class of all direct
summands of direct sums of the form $\bigoplus_{i\in I}X_i$, where $I$
is any set (resp.\ finite set) and $X_i\in\Xx$ for all $i$. By
$\Gen(\Xx)$ we denote the class of all objects $Y$ \emph{generated by}
$\Xx$, that is, such that there is an epimorphism $X\ra Y$ with
$X\in\Add(\Xx)$ (and similarly $\gen(\Xx)$ with $\add(\Xx)$). As usual
we write $X^{(I)}$ for $\bigoplus_{i\in I}X$.\medskip

By $\Prod(\Xx)$ we denote the class of all direct summands of products
of the form $\prod_{i\in I}X_i$, where $I$ is any set and $X_i\in\Xx$
for all $i$. By $\Cogen(\Xx)$ we denote the class of all objects $Y$
\emph{cogenerated by} $\Xx$, that is, such that there is a
monomorphism $Y\ra X$ with $X\in\Prod(\Xx)$. We write $X^I$ for
$\prod_{i\in I}X$.\medskip

We denote by $\varinjlim{\Xx}$ the direct limit closure of $\Xx$ in
$\Aa$. We will often use also the shorthand notation
$\vec{\Xx}=\varinjlim{\Xx}$.\medskip

Let $(I,\leq)$ be an ordered set and $\Xx_i$ classes of objects for
all $i\in I$, in any additive category.  We write
$\bigvee_{i\in I}\Xx_i$ for $\add(\bigcup_{i\in I}\Xx_i)$ if
additionally $\Hom(\Xx_j,\Xx_i)=0$ for all $i<j$ is satisfied. In
particular, notation like $\Xx_1\vee\Xx_2$ and
$\Xx_1\vee\Xx_2\vee\Xx_3$ makes sense (where $1<2<3$).

\section{Pure-injectivity}\label{sec:purity}
The notion of purity is of great importance in our setting. For
details we refer to~\cite{prest:2011,crawley-boevey:1994}.  Let $\Aa$
be an abelian category.  We denote the full subcategory of
finitely presented objects in $\Aa$ by $\fp(\Aa)$.
\begin{itemize}
\item We say that $\Aa$ is \emph{Grothendieck} if all set-indexed
  coproducts exist, direct limits are exact and $\Aa$ has a generator.
\item We say that $\Aa$ is \emph{locally finitely presented} if
  $\fp(\Aa)$ is skeletally small and every object in $\Aa$ is a direct
  limit of objects in $\fp(\Aa)$.
\item We say that $\Aa$ is \emph{locally coherent} if $\Aa$ is locally
  finitely presented and $\fp(\Aa)$ is abelian.
\item If $\Aa$ is $k$-linear over a field $k$, then $\Aa$ is called
  \emph{Hom-finite} if $\Hom_\Aa(C,D)$ is a finite-dimensional
  $k$-vector space for every pair of objects $C$ and $D$ in $\Aa$.
\item Let $\Aa$ be $k$-linear locally coherent and
  $\D:=\Hom_k(-,k)$. Then $\fp(\Aa)$ is said to satisfy \emph{Serre
    duality} if $\fp(\Aa)$ is Hom-finite and if there is an
  autoequivalence $\tau\colon\fp(\Aa)\ra\fp(\Aa)$ and an isomorphism
  $\D\Ext^1_{\Aa}(X,Y)=\Hom_{\Aa}(Y,\tau X)$, natural in
  $X,\,Y\in\fp(\Aa)$. Moreover, $\Aa$ is said to satisfy
  \emph{(generalised) Serre duality} if additionally
  $\D\Ext^1_{\Aa}(X,Y)=\Hom_{\Aa}(Y,\tau X)$ and
  $\Ext^1_{\Aa}(Y,\tau X)=\D\Hom_{\Aa}(X,Y)$ hold for all objects
  $Y\in\Aa$, $X\in\fp(\Aa)$.
\end{itemize}
\begin{remark} Since we have assumed that $\Aa$ is abelian, we have
  that, if $\Aa$ is locally finitely presented, then $\Aa$ is
  Grothendieck.  See, for example,
  \cite[Sec.~2.4]{crawley-boevey:1994}.
\end{remark}
\begin{definition} Let $\Aa$ be a locally finitely presented abelian
  category.
  \begin{enumerate}
  \item An exact sequence
    $0\ra A\stackrel{\alpha}\ra B\stackrel{\beta}\ra C\ra 0$ in $\Aa$
    is called \emph{pure-exact}, if for every $F\in\fp(\Aa)$ the
    induced sequence
\[ 0\ra\Hom_\Aa(F,A)\ra\Hom_\Aa(F,B)\ra\Hom_\Aa(F,C)\ra 0\] is exact.
In this case $\alpha$ (resp.~$\beta$) is called a \emph{pure
  monomorphism} (resp.~\emph{pure epimorphism}), and $A$ a \emph{pure
  subobject} of $B$.

\item A \emph{pure-essential morphism} is a pure-monomorphism $j$ in
  $\Aa$ such that, if $fj$ is a pure monomorphism for some morphism
  $f$ in $\Aa$, then $f$ is a pure monomorphism.

\item An object $E\in\Aa$ is called \emph{pure-injective} if for every
  pure-exact sequence $0\ra A\ra B\ra C\ra 0$ the induced sequence
\[0\ra\Hom_\Aa(C,E)\ra\Hom_\Aa(B,E)\ra\Hom_\Aa(A,E)\ra 0\] is exact.

\item For an object $M$ in $\Aa$, a \emph{pure-injective envelope} of
  $M$ is a pure-essential morphism $M\to N$ where $N$ is
  pure-injective.

\item An object $N$ is called \emph{superdecomposable} if $N$ has no
  nonzero indecomposable direct summands.

\item An object $E\in\Aa$ is called $\Sigma$-\emph{pure-injective} if
  the coproduct $E^{(I)}$ is pure-injective for every set $I$.

\item An object $Y\in\Aa$ is called \emph{fp-injective} if
    $\Ext^1_{\Aa}(X,Y)=0$ for every $X\in\fp(\Aa)$.
\end{enumerate}
\end{definition}
For every locally finitely presented abelian category $\Aa$, there
exists a locally coherent Grothendieck category $\Ff(\Aa)$ and a fully
faithful functor $d\colon \Aa \to \Ff(\Aa)$ that identifies the
pure-exact sequences in $\Aa$ with exact sequences in $\Ff(\Aa)$ and
the pure-injective objects in $\Aa$ with the injective objects in
$\Ff(\Aa)$; see \cite{crawley-boevey:1994, Krause2:1998}.  The
pure-injective objects in $\Aa$ therefore inherit the following
properties of injective objects in $\Ff(\Aa)$.
\begin{proposition}\label{prop:pure-inj-normal-form}
Let $\Aa$ be a locally finitely presented abelian category.  The
following statements hold.
\begin{enumerate}
\item Every object $M$ in $\Aa$ has a pure injective envelope
  $M \to \PE(M)$ that is unique up to isomorphism.
\item Every pure-injective object $N$ has the following form
  \[ N \cong \PE\left(\bigoplus_{i\in I}N_i\right) \oplus N_c\] where
  $\{N_i\}_{i\in I}$ is the set of indecomposable pure-injective
  summands of $N$ and $N_c$ is superdecomposable.
\item Let $N$ be a pure-injective object and suppose
  \[ N \cong \PE\left(\bigoplus_{i\in I}N_i\right) \oplus N_c \cong
    \PE\left(\bigoplus_{j\in J}M_j\right) \oplus M_c \] such that
  $N_i, M_j$ are indecomposable for all $i\in I$, $j\in J$ and $N_c$,
  $M_c$ are superdecomposable.  Then there exists a bijection
  $\sigma \colon I\to J$ such that $N_i \cong M_{\sigma(i)}$ for all
  $i\in I$ and $N_c \cong M_c$.
\end{enumerate}
\end{proposition}
\begin{proof}
  Both (1) and (2) follow from the analogous result for injective
  objects in a Grothendieck category; see, for example,
  \cite[Prop.~X.2.5, Cor.~X.4.3]{stenstroem:1975} for (1) and
  \cite[Thm.~E.1.9]{prest:2009} for (2) and (3).
\end{proof}
If $N$ is a pure-injective object as in Proposition
\ref{prop:pure-inj-normal-form} such that $N_c=0$, we say that $N$ is
a \emph{discrete} pure-injective object. The following statement
provides an alternative characterisation of pure-injectivity; it is
often called the Jensen-Lenzing criterion.
\begin{proposition}[{\cite[Thm.~5.4]{prest:2011}}]\label{prop:jensen-lenzing-crit}   
  An object $E$ in a locally finitely presented abelian category
  $\Aa$ is pure-injective if and only if for any index set $I$ the
  summation morphism $E^{(I)}\ra E$ factors through the canonical
  embedding $E^{(I)}\ra E^I$. \qed
\end{proposition}
\begin{lemma}\label{lem:local-end-ring}
  Every indecomposable pure-injective object in $\Aa$ has a local
  endomorphism ring. 
\end{lemma}
\begin{proof}
  Cf.\ \cite[Cor.~7.5]{jensen:lenzing:1989}.
\end{proof}
\begin{remark}
  In this article, pure-injectivity in a compactly generated
  triangulated category will be defined by the property given in
  Proposition \ref{prop:jensen-lenzing-crit}. Moreover, Lemma
  \ref{lem:local-end-ring} is true in such a category.
\end{remark}
\begin{lemma}\label{lem:pure-injective}
  Assume that $\Aa$ is a locally coherent abelian $k$-category
  over a field $k$ and that $\fp(\Aa)$ is Hom-finite. Then every
  object $F\in\fp(\Aa)$ is $\Sigma$-pure-injective.
\end{lemma}
\begin{proof}
  By Hom-finiteness of $\fp(\Aa)$ this follows directly
  from~\cite[(3.5)~Thm.~2]{crawley-boevey:1994}.
\end{proof}
\subsection*{Pure-injectives and Ext}
If $E$ is a pure-injective module over a ring $R$, then the functor
$\Ext^1_R(-,E)$ sends direct limits to inverse limits.  Here we show
that pure-injective objects in our setting have a similar property.
\begin{theorem}\label{thm:ext-direct-limit-pure-inj}
  Let $\Aa$ be a locally finitely presented Grothendieck category.
  Then, for any pure-injective object $E$ in $\Aa$ and any directed
  system of objects $M_i$ ($i\in I$), we have that
  $\Ext^1_\Aa(\varinjlim M_i, E) =0$ whenever $\Ext^1_\Aa(M_i, E) = 0$
  for all $i\in I$.
\end{theorem}
\begin{proof}
  The exact sequence
  $0 \to K \overset{f}{\to} \coprod_{i\in I} M_i \overset{g}{\to}
  \varinjlim M_i \to 0$ where $K \cong \Ker{g}$ is pure-exact since
  finitely presented objects in $\Aa$ commute with direct limits and
  direct sums.  Therefore, any morphism $K \to E$ factors through $f$.
  So if we apply $\Hom_\Aa(-,E)$, we obtain the exact sequence
  \[ \dots \Hom_\Aa(\coprod M_i, E) \to \Hom_\Aa(K, E)
    \overset{0}{\to} \Ext_\Aa^1(\varinjlim M_i, E) \to \prod
    \Ext^1_\Aa(M_i, E) \to \dots .\] By our assumption, we may
  conclude that $\Ext_\Aa^1(\varinjlim M_i, E) = 0$.
\end{proof}
\begin{numb}\label{numb:loc-noeth-comp-gen}
  If $\Aa$ is a locally noetherian Grothendieck category (that is, a
  Grothendieck category which has a family of noetherian generators)
  such that every object in $\Aa$ has finite injective dimension, it
  follows from~\cite[Prop.~2.3, Ex.~3.10]{krause:2005b} that the
  derived category $\Derived{\Aa}$ is compactly generated and the full
  subcategory of compact objects coincides with $\bDerived{\fp(\Aa)}$.
  There is a well-developed notion of pure-injectivity in a compactly
  generated triangulated category; we refer to \cite{krause:2000} for
  some background.  We will make use of the interaction between the
  purity in $\Aa$ and the purity in $\Derived{\Aa}$.  In particular,
  we note that, by \cite[Thm.~1.8]{krause:2000}, an object $E$ in a
  compactly generated triangulated category is pure-injective if and
  only if, for every index set $I$, the summation morphism
  $E^{(I)} \rightarrow E$ factors as in
  Proposition~\ref{prop:jensen-lenzing-crit}.
\end{numb}
In the following lemma we will need to distinguish between products
taken in a Grothendieck category $\Aa$ and products taken in the
derived category $\Derived{\Aa}$.  We will denote an $S$-indexed
direct sum of copies of $E$ in $\Aa$ by $\bigoplus^\Aa_{S} E$ and an
$S$-indexed direct products of copies of $E$ in $\Aa$ by
$\prod^\Aa_{S} E$.  Similar notation will be used for direct sums and
products taken in $\Derived{\Aa}$.
\begin{lemma}\label{lem:pure-in-der-cat}
  Let $\Aa$ be a locally noetherian hereditary Grothen\-dieck
  category, i.e.~$\Ext^2(-,-) =0$.  Then $E$ is pure-injective in
  $\Aa$ if and only if $E$ is pure-injective when considered as an
  object of $\Derived{\Aa}$.
\end{lemma}
\begin{proof}
If $E$ is pure-injective in $\Derived{\Aa}$, it follows from
\cite[Prop.~5.2]{laking:2018} that $E$ is pure-injective in $\Aa$.

For the converse, let $E$ be pure-injective in $\Aa$, and let
$f^\Aa \colon \bigoplus^\Aa_{S} E \rightarrow \prod^\Aa_{S} E$ be the
canonical embedding and
$\Sigma^\Aa \colon \bigoplus^\Aa_{S} E \rightarrow E$ be the summation
morphism for some set $S$.  By Proposition
\ref{prop:jensen-lenzing-crit}, there exists a morphism
$h^\Aa \colon E^S \rightarrow E$ such that $h^\Aa f^\Aa = \Sigma^\Aa$.

  Now consider the $S$-indexed direct product $\prod^{\Tt}_{S} E$
  taken in $\Tt := \Derived{\Aa}$.  Observe that we have an
  isomorphism
  $\prod^{\Tt}_{S} E \cong H^0(\prod^{\Tt}_{S} E) \oplus
  {H^1(\prod^{\Tt}_{S} E)[-1]}$ (see, for example,
  \cite[Sec.~1.6]{krause:2007}) and we also have that
  $H^0(\prod^{\Tt}_{S} E) \cong \prod^\Aa_S E$.  Clearly we have
  $\bigoplus^\Aa_{S} E \cong \bigoplus^{\Tt}_{S} E$ because coproducts
  are exact in $\Aa$.

  Let
  $f^{\Tt}_S \colon \bigoplus_S^{\Tt} E {\longrightarrow}
  \prod^{\Tt}_S E$ be the canonical morphism from the coproduct to the
  product.  Since
  $\Hom_\Tt(\bigoplus_S^\Tt E, H^1(\prod^{\Tt}_{S} E)[-1]) = 0$, we
  have that the projection
  $\pi \colon \prod_S^\Tt E \to H^0(\prod^{\Tt}_{S} E) \cong
  \prod^\Aa_S E$ induces an isomorphism
  \[\Hom_\Tt(\bigoplus_S^\Tt E, \prod_S^\Tt E)
    \overset{\sim}{\longrightarrow} \Hom_\Tt(\bigoplus_S^\Tt E,
    \prod_S^\Aa E).\] Moreover, the universal properties of the
  canonical morphisms ensure that this isomorphism sends $f_S^\Tt$ to
  $f^\Aa_S$, that is $f^\Aa_S = \pi \circ f_S^\Tt$.  Therefore, we
  have
  $(h^\Aa \circ \pi) \circ f^\Tt_S = \Sigma_S^\Aa = \Sigma_S^\Tt$,
  which shows that the object $E$ is pure-injective in
  $\Derived{\Aa}$.
\end{proof}
For a compactly generated triangulated category $\Tt$, we denote the
full subcategory of compact objects in $\Tt$ by $\Tt^c$.  An important
tool for studying $\mathcal{T}$ is the category $\Mod{\text{-}\Tt^c}$
of additive functors from $(\Tt^c)^{\op}$ to the category $\Ab$ of
abelian groups; see, for example, \cite[Sec.~1.2]{krause:2005b}.  We
make use of the \emph{restricted Yoneda functor}
$\y \colon \Tt \to \Mod{\text{-}\Tt^c}$, which takes an object $M$ in
$\Tt$ to the functor $\y (M) := \Hom_\Tt(-,M)|_{\Tt^c}$.  An object
$E$ in $\Tt$ is pure-injective if and only if $\y(E)$ is injective in
$\Mod{\Tt^c}$; see \cite[Thm.~1.8]{krause:2000}.
\subsection*{Pure subobjects of products of compact objects}
Let $k$ be a field and let $\Tt$ be a compactly generated triangulated
$k$-linear category for a field $k$. We will denote the category of
additive functors from $\Tt^c$ to $\Ab$ by ${\Tt^c}\text{-}\Mod$.
Since $\Tt$ is $k$-linear, we have a functor
$D \colon \Mod\text{-}\Tt^c \ra \Tt^c\text{-}\Mod$ given by
postcomposition with $\Hom_k(-,k)$.  Similarly, we have
$D \colon \Tt^c\text{-}\Mod \rightarrow \Mod\text{-}\Tt^c$.

We say that $\Tt^c$ has \emph{Auslander-Reiten triangles} if for every
indecomposable object $C$ in $\Tt^c$, there exist objects $A,B,D,E$
and Auslander-Reiten triangles
\[C \to D \to E \to C[1]\quad\text{ and }\quad A \to B \to C \to
  A[1]\] in $\Tt^c$.  The definition of an Auslander-Reiten triangle
can be found in \cite[Sec.~4.1]{Happel:1988}.

\begin{lemma}[{\cite[Lem.~4.1, Thm.~4.4]{krause:2005}}]\label{lem:Serre} 
  Suppose $\Tt$ is a compactly generated $k$-linear triangulated
  category such that $\Tt^c$ is Hom-finite.  Then the following
  statements hold. \begin{enumerate}
  \item There is a functor $T \colon \Tt^c \rightarrow \Tt$, together
    with a natural isomorphism
    $$\D\Hom_\Tt(C, X) \cong \Hom_\Tt(X, TC)$$ for every compact object
    $C$ and every object $X$ in $\Tt$.
  \item The functor $T$ restricts to an equivalence
    $\Tt^c \rightarrow \Tt^c$ if and only if $\Tt^c$ has
    Auslander-Reiten triangles.
\end{enumerate}
\end{lemma}

\begin{remark}\label{rem:compact-pure-inj}
  For a compactly generated triangulated $k$-linear category where
  $\Tt^c$ is Hom-finite, every compact object is endofinite and hence
  pure-injective by \cite[Thm.~1.2]{Krause:1999}.
\end{remark}

\begin{proposition}\label{prop:class-of-pure-injective-complexes}
  Let $k$ be a field and let $\Tt$ be a compactly generated
  triangulated $k$-linear category such that $\Tt^c$ is Hom-finite
  and has Auslander-Reiten triangles.  Then every object in
  $\Tt$ is a pure subobject of a product of compact objects.
\end{proposition}
\begin{proof}
  Let $F := \Hom_\Tt(-, N)|_{\Tt^c}$ for some object $N$ in $\Tt$.
  Then $DF$ is cohomological so, by \cite[Lem.~2.7]{krause:2000} or
  \cite[Rem.~8.12]{beligiannis:2000}, it is a flat object of
  $\Tt^c\text{-}\Mod$.  By \cite[Thm.~3.2]{oberst:rohrl:1970}, there
  exists a directed system of representable functors
  $\{\Hom_{\Tt^c}(C_i, -)\}_{i\in I}$ such that
  $DF \cong \varinjlim_{i\in I} \Hom_{\Tt}(C_i, -)$.  It follows
  that there exists a canonical epimorphism
  $\bigoplus_{i\in I} \Hom_{\Tt}(C_i, -) \rightarrow DF \rightarrow
  0$.  Applying the functor $D$ again we obtain a monomorphism
  \[ 0 \rightarrow D^2F \rightarrow \prod_{i\in I} D\Hom_{\Tt}(C_i,
    -).  \] Now, applying Lemma~\ref{lem:Serre}, we have that
  \[\prod_{i\in I} D\Hom_{\Tt}(C_i, -) \cong \prod_{i\in I}
    \Hom_{\Tt}(-, TC_i) \cong \Hom_{\Tt}(-, \prod_{i\in I}TC_i).\]
  Moreover we have a natural family of monomorphisms from each vector
  space to its double dual and this induces a monomorphism
  $0 \rightarrow F \rightarrow D^2F$.  Composing these morphisms we
  obtain a monomorphism
  $0 \rightarrow F \rightarrow \Hom_{\Tt}(-, \prod_{i\in I}TC_i)$.
  Finally, since $\prod_{i\in I}TC_i$ is pure-injective, we have that
  this is induced by a pure monomorphism
  $N \rightarrow \prod_{i\in I}TC_i$ by \cite[Thm.~1.8]{krause:2000}.
\end{proof}

\subsection*{Compact summands of products}

In Remark \ref{rem:compact-pure-inj} we have that, when $\Tt$ is
compactly generated with $\Tt^c$ Hom-finite, every compact object $C$
is pure-injective.  In the next proof we show that if $\Tt^c$ also has
Auslander-Reiten triangles, then $\y(C)$ is the injective envelope of
a simple functor in $\Mod{\text{-}\Tt^c}$.  In particular, compact objects
have the following property with respect to products of pure-injective
objects in $\Tt$.

\begin{proposition}\label{prop:summands-of-products}
  Let $\Tt$ be a compactly generated triangulated category such that
  $\Tt^c$ is Hom-finite and has Auslander-Reiten triangles.  If a
  compact object $C$ is a direct summand of a product
  $\prod_{i\in I} N_i$ of pure-injective objects $\{N_i\}_{i\in I}$ in
  $\Tt$, then $C$ is a direct summand of $N_i$ for some $i\in I$.
\end{proposition}
\begin{proof}
  Let $C\overset{f}{\to} D\to E\to C[1]$ be an Auslander-Reiten
  triangle and consider the functor $F := \Ker(\y(f))$.

First we show that $F$ is a simple functor.  Consider the category
$\mathrm{Coh}(\Tt)$ of coherent functors $\Tt \to \Ab$ i.e.~covariant
functors that are of the form $\Coker(\Hom_\Tt(g,-))$ for some
morphism $g$ in $\Tt^c$.  In \cite{krause:2002}, Krause shows that
there exists a duality
$(-)^\vee \colon \mod{\text{-}\Tt^c} \to \mathrm{Coh}(\Tt)$ where
$G^\vee(X) := \Hom_{\Mod{\Tt^c}}(G, \y(X))$ for each functor $G$ in
$\mod{\text{-}\Tt^c}$ and object $X$ in $\Tt$.  By
\cite[Cor.~1.12]{Arnesen:Laking:Paulsztello:Prest:2017}, we have that
the functor $F^\vee$ is isomorphic to $\Coker(\Hom_\Tt(f,-))$ which is
a simple functor.  It follows that $F$ is a simple functor in
$\Mod{\text{-}\Tt^c}$.

By assumption, there is a split monomorphism
$C \to \prod_{i\in I} N_i$ and so its image
$\y(C) \to \prod_{i\in I} \y(N_i)$ is a split monomorphism in
$\Mod{\text{-}\Tt^c}$.  Since $\y(C)$ is an indecomposable injective object,
the monomorphism $F \to \y(C)$ must be an injective envelope and $F$
must be essential in $\y(C)$.  It follows that the composition
$F \to \y(C)\to \prod_{i\in I} \y(N_i) \to \y(N_i)$ is a non-zero
monomorphism $F \to \y(N_i)$ for some $i\in I$.  But then the
injective envelope $\y(C)$ of $F$ must be a direct summand of the
injective object $\y(N_i)$.  By \cite[Thm.~1.8]{krause:2000} we have
that $C$ is a direct summand of $N_i$.
\end{proof}

\section{Cotilting objects}\label{sec:cotilting}
Let $\Aa$ be a Grothendieck category.
\begin{definition}[{\cite[Def.~2.4]{coupek:stovicek:2019}}]\label{def:cotilting}
  An object $C\in\Aa$ is called a \emph{cotilting object} if
  $\Cogen(C)=\lperpe{C}$, and if this class contains a generator for
  $\Aa$. Then $\Cogen(C)$ is called the associated \emph{cotilting
    class}.
\end{definition}
\begin{lemma}[{\cite[Thm.~2.11]{coupek:stovicek:2019}}]
  An object $C\in\Aa$ is cotilting if and only if the following
  conditions are satisfied:
\begin{enumerate}
    \item[(CS0)] $C$ has injective dimension $\id(C)\leq 1$.\smallskip
    \item[(CS1)] $\Ext^1(C^I,C)=0$ for every cardinal $I$.\smallskip
    \item[(CS3)] For every injective cogenerator $W$ of $\QHh$ there
      is a short exact sequence $$0\ra C_1 \ra C_0 \ra W\ra 0$$ with
      $C_0,\,C_1\in\Prod(C)$.
    \end{enumerate}
    Each cotilting $C$ moreover satisfies
    \begin{enumerate}
    \item[(CS2)] $\lperp{C}=0$, that is: if $X\in\QHh$ satisfies
      $\Hom(X,C)=0=\Ext^1(X,C)$, then $X=0$. \qed
    \end{enumerate}
\end{lemma}
We used this order of numbering since (CS0), (CS1) and (CS2) are the
duals of the corresponding properties (TS0), (TS1), (TS2) for tilting
sheaves in~\cite{angeleri:kussin:2017}.
\begin{theorem}[{\cite[Thm.~3.9]{coupek:stovicek:2019}}]\label{thm:cotilting-is-pure-injective}
  Let $C\in\Aa$ be cotilting. Let $\Ff=\lperpe{C}=\Cogen(C)$ be the
  associated cotilting class. Then $C$ is pure-injective and $\Ff$ is
  closed under direct limits in $\Aa$. \qed
\end{theorem}
It follows that Definition~\ref{def:cotilting} is equivalent to the
definition of cotilting objects given in~\cite{buan:krause:2003} for
locally noetherian Grothendieck categories. The following is
well-known and easy to show.
\begin{lemma}\label{lem:properties-cotilting}
  Let $C\in\Aa$ be cotilting with associated cotilting class $\Ff$.
  \begin{enumerate}
  \item $\Ff=\Copres(C)$, the class of objects in $\Aa$ which are
    kernels of morphisms of the form $C^{I}\ra C^{J}$.
  \item $\Ff\cap\rperpe{\Ff}=\Prod(C)$. \qed
  \end{enumerate}
\end{lemma}
\begin{corollary}\label{cor:properties-cotilting}
  Let $\Aa$ be locally noetherian with the property that every object
  in $\Aa$ has finite injective dimension.  Let $C\in\Aa$ be a
  cotilting object with cotilting class $\Ff=\lperpe{C}$. If
  $B\in\fp(\Aa)$ is indecomposable with $B\in\Ff\cap\rperpe{\Ff}$,
  then $B$ is a direct summand of $C$. \qed
\end{corollary}
\begin{proof}
  By Lemma~\ref{lem:properties-cotilting}, we have that
  $\Ff \cap \rperpe{\Ff} = \Prod(C)$ and so we may apply
  \cite[Cor.~2.13]{coupek:stovicek:2019} to obtain that products of
  copies of $C$ in $\Aa$ coincide with products of copies of $C$ in
  $\Derived{\Aa}$.  By Theorem~\ref{thm:cotilting-is-pure-injective},
  the cotilting object $C$ is pure-injective in $\Aa$ so, by
  Lemma~\ref{lem:pure-in-der-cat}, it is also pure-injective in
  $\Derived{\Aa}$.  Finally, we may apply
  Proposition~\ref{prop:summands-of-products}, to obtain that $B$ is a
  direct summand of $C$ in $\Derived{\Aa}$ and hence in $\Aa$.
\end{proof}
\begin{definition}
  \begin{enumerate}
  \item Two cotilting objects $C$, $C'\in\Aa$ are \emph{equivalent}, if
  they have the same cotilting class. This is equivalent to
  $\Prod(C)=\Prod(C')$.
\item A cotilting object $C\in\Aa$ is called \emph{minimal} if, for
  any other cotilting object $C'$ with same cotilting class
  $\Cogen(C')=\Cogen(C)$, we have that $C$ is a direct summand of
  $C'$.
  \end{enumerate}
\end{definition}
\emph{Let $\Aa$ additionally be locally finitely presented with
  $\Aa_0=\fp(\Aa)$.}
\begin{theorem}[{\cite[Thm.~3.13]{coupek:stovicek:2019}, \cite[Thm.~1.13]{buan:krause:2003}}]\label{thm:cotilting-is-of-finite-type}
  Let $\Aa$ be locally noetherian. The torsionfree classes $\Ff$ in
  $\Aa$ associated to a cotilting object bijectively correspond to the
  torsion pairs $(\Tt_0,\Ff_0)$ in $\Aa_0$ where $\Ff_0$ is a
  generating class for $\Aa_0$. The correspondence is given by
  $$\Ff\mapsto\Ff\cap\Aa_0\quad\text{and}\quad
  (\Tt_0,\Ff_0)\mapsto\varinjlim(\Ff_0).$$ Accordingly, two cotilting
  objects $C,\,C'\in\Aa$ are equivalent if and only if
  $\lperpe{C}\cap\Aa_0=\lperpe{C'}\cap\Aa_0$. \qed
\end{theorem}
\begin{definition}
  A cotilting object $C\in\Aa$ is called \emph{large} if it is not
  equivalent to a coherent cotilting object.
\end{definition}

\begin{definition}
  Let $E\in\Aa$.
  \begin{enumerate}
  \item $E$ is called \emph{rigid}, if $\Ext^1(E,E)=0$.
  \item $E$ is called \emph{self-orthogonal}, if
    $\Ext^1(E^{\alpha},E)=0$ for every cardinal $\alpha$.
  \item A self-orthogonal $E$ is called \emph{maximal self-orthogonal}, if
    $\Prod(E)\subseteq\Prod(F)$ implies $\Prod(E)=\Prod(F)$ for every
    self-orthogonal $F$.
  \end{enumerate}
\end{definition}
\begin{proposition}
  Let $\Aa$ be hereditary. Every cotilting object in $\Aa$ is maximal
  self-orthogonal.
\end{proposition}
\begin{proof}
  The proof in~\cite[Prop.~3.1]{buan:krause:2003} does also work in
  this more general situation.
\end{proof}
\begin{proposition}
  Let $\Aa$ be locally noetherian Grothendieck and hereditary with
  $\Aa_0=\fp(\Aa)$. An object $C$ in $\Aa$ is cotilting if and only if
  (CS1) and (CS2) hold and $\lperpe{C}\cap\Aa_0$ is generating.
  \end{proposition}
\begin{proof}
  Let $C$ satisfy (CS1) and (CS2). It is sufficient to show that
  $\Cogen(C)=\lperpe{C}$. By (CS1) and since $\Aa$ is hereditary we
  easily get $\Cogen(C)\subseteq\lperpe{C}$. For the reverse
  inclusion, we let $X\in\lperpe{C}$ and consider the short exact
  sequences induced by the reject $K$ of $\{C\}$ in $X$, that is,
  $0\ra K\ra X\ra U\ra 0$ and $0\ra U\ra C^I\ra Y\ra 0$ with
  $I=\Hom(X,C)$. By applying $\Hom(-,C)$ to these sequences and using
  again that $\Aa$ is hereditary, we obtain
  $\Ext^1(K,C)=0=\Ext^1(U,C)\cong\Hom(K,C)$, and then $K=0$ by
  (CS2). We get $X\in\Cogen(C)$.
\end{proof}
\subsection*{Cotilting objects and injective cogenerators}
Let $\Aa$ be a Grothendieck category and $(\Tt, \Ff)$ a torsion pair
in $\Aa$. We define a t-structure $(\Uu_\Tt, \Vv_\Ff)$ on
$\Derived{\Aa}$ as follows:
\[ \Uu_\Tt := \{ X \in \Derived{\Aa}\mid H^i(X) = 0 \text{ for all }
  i> 0 \text{ and } H^0(X) \in \Tt \}\]
\[ \Vv_\Ff := \{ X \in \Derived{\Aa}\mid H^i(X) = 0 \text{ for all }
  i< -1 \text{ and } H^{-1}(X) \in \Ff \}.\]

\noindent We call $(\Uu_\Tt, \Vv_\Ff)$ the \emph{HRS-tilted
  t-structure} of $(\Tt, \Ff)$,
after~\cite{happel:reiten:smaloe:1996}.

The following full subcategory 
\[ \Gg = \{ X \in \Derived{\Aa} \mid H^{-1}(X) \in \Ff, H^0(X) \in \Tt
  \text{ and } H^i(X) = 0 \text{ for } i \neq 0,\,-1 \}\] of
$\Derived{\Aa}$ is the \emph{heart} of the t-structure
$(\Uu_\Tt, \Vv_\Ff)$. It is sometimes also called an \emph{HRS-tilt}
of $\Aa$. Then $(\Ff[1],\Tt)$ is a torsion pair in $\Gg$, and if
$(\Tt,\Ff)$ is a cotilting torsion pair, then
$\Derived{\Gg}=\Derived{\Aa}$,
cf.~\cite{stovicek:kerner:trlifaj:2011}.

We will also consider (in Section~\ref{sec:Euler-zero}) the category
$\Gg[-1]$ instead and call it the \emph{$[-1]$-shifted heart}. In this
case we have that $(\Ff,\Tt[-1])$ is a torsion pair in
$\Gg[-1]$. Since $[-1]$ is an automorphism of $\Derived{\Aa}$ there is
just a notational difference (which has some tradition in the theory
of weighted projective lines).

Besides the mentioned results
from~\cite{coupek:stovicek:2019} we will also need the following:
\begin{theorem}[{\cite[Thm.~5.2]{saorin:2017}}]\label{thm:loc-coh-hearts}
  Let $\Aa$ be a locally noetherian Grothendieck category and let
  $(\Tt, \Ff)$ be a torsion pair in $\Aa$. Then $\Gg$ is a locally
  coherent Grothendieck category (with $\Gg \cap \bDerived{\fp\Aa}$
  the class of finitely presented objects) if and only if $\Ff$ is
  closed under direct limits. \qed
\end{theorem}
\begin{theorem}[{\cite[Prop.~4.4]{coupek:stovicek:2019}}]
  Let $\Aa$ be a Grothendieck category and let $(\Tt, \Ff)$ be a
  torsion pair where $\Ff$ is a generating class.  Then an object $E$
  in $\Gg$ is an injective cogenerator of $\Gg$ if and only if
  $E \cong C[1]$ where $C$ is a cotilting object in $\Aa$ with
  $\Ff = \Cogen(C)$. \qed
\end{theorem}
\noindent Together we obtain the following corollary:
\begin{corollary}\label{Cor: loc coh heart plus min cotilt}
Let $\Aa$ be a locally noetherian category.  Then the following
statements hold:
\begin{enumerate}
\item If $(\Tt, \Ff)$ is a cotilting torsion pair in $\Aa$, then $\Gg$
  is locally coherent. 
\item If $(\Tt, \Ff)$ is a cotilting torsion pair in $\Aa$ and $\Gg$
  has a minimal injective cogenerator, then there exists a minimal
  cotilting object $C$ in $\Aa$ with $\Cogen(C) = \Ff$. \qed
\end{enumerate}
\end{corollary}
\subsection*{$\Sigma$-pure injective cotilting objects}
Before we continue the discussion on minimal cotilting objects, we
observe that we obtain the following criterion as a corollary of the
above. An analogous result for modules over any ring can be found
in~\cite[Thm.~5.3]{colpi:mantese:tonolo:2010}. This is also shown in a
more general setting in~\cite[Prop.~5.6]{laking:2018}.
\begin{corollary}
  Let $\Aa$ be a locally coherent Grothendieck category and let
  $(\Tt, \Ff)$ be a torsion pair with $\Ff=\lperpe{C}$ associated with
  a cotilting object $C\in\Aa$. The following are equivalent:
  \begin{enumerate}
  \item[(1)] $C$ is $\Sigma$-pure-injective in $\Aa$.
  \item[(2)] $\Gg$ is locally noetherian.
  \end{enumerate}
\end{corollary}
\begin{proof}
  By the preceding discussion, $\Gg$ is locally coherent and $C$ is
  pure-injective. $\Prod(C[1])$ (in $\Gg$) is the class of injective
  objects in $\Gg$.

  By~\cite[Prop.~V.4.3]{stenstroem:1975}, $\Gg$ is locally noetherian
  if and only if each coproduct of injective objects is injective,
  that is, $\Prod(C[1])$ in $\Gg$ is closed under
  coproducts. By~\cite[Cor.~2.13]{coupek:stovicek:2019} this is
  equivalent to $\Prod(C)$ in $\Aa$ being closed under coproducts. If
  this holds then in particular $C^{(I)}$ is pure-injective for each
  set $I$, that is, $C$ is $\Sigma$-pure-injective. Conversely, if $C$
  is $\Sigma$-pure-injective, then
  by~\cite[(3.5)~Thm.~2]{crawley-boevey:1994} so is each object in
  $\Prod(C)$ and is a coproduct of indecomposables. It follows that
  each injective object in $\Gg$ is a coproduct of indecomposable
  objects. Thus $\Gg$ is locally noetherian
  by~\cite[Thm.~A.11]{krause:2001}.
\end{proof}

\subsection*{Locally finitely generated Grothendieck categories and minimal injective cogenerators}
Let $\Aa$ be a Grothendieck category.  An object $F$ in $\Aa$ is
\emph{finitely generated} if, whenever $F = \sum_{i\in I} F_i$ for a
direct family of subobjects $\{F_i\}_{i\in I}$ of $F$, there exists an
index $i_0 \in I$ such that $F = F_{i_0}$.

\begin{remark}
We define $\sum$ as follows (\cite[pg.~88]{stenstroem:1975}).  Let
$\{C_i\}_{i\in I}$ be a family of subobjects of $C$, then the
monomorphisms $C_i \to C$ induce a morphism $\alpha \colon
\bigoplus_{i\in I}C_i \to C$.  The image of $\alpha$ is denoted
$\sum_{i\in I}C_i$ and is called the \emph{sum} of the subobjects
$\{C_i\}_{i\in I}$.  By \cite[Ch.~IV, Ex.~8.3]{stenstroem:1975}, if
$\{C_i\}_{i\in I}$ is a direct family, then $\varinjlim_{i\in I}C_i$
is a subobject of $C$ and coincides with $\sum_{i\in I} C_i$. 
\end{remark}

We say that $\Aa$ is \emph{locally finitely generated} if there
exists a family of finitely generated generators.  By
\cite[Lem.~3.1(i)]{stenstroem:1975}, if $C$ is finitely generated,
then the image of a morphism $C \to D$ is finitely generated.  Thus
$\Aa$ is locally finitely generated if and only if, for every object
$C$ in $\Aa$, there is a direct family $\{C_i\}_{i\in I}$ of finitely
generated subobjects of $C$ such that $C = \sum_{i\in I}C_i$.

\begin{proposition}[{Element-free version of
    \cite[Prop.~6.6]{stenstroem:1975}}] Let $\Aa$ be a locally
  finitely generated Grothen\-dieck category.  Then an injective
  object $E$ is a cogenerator if and only if it contains as a
  subobject an isomorphic copy of each simple object.
\end{proposition} 
\begin{proof}
If $E$ is a cogenerator, then there exists a non-zero morphism $S \to
E$ for each simple object $S$ which is necessarily a monomorphism. 

For the converse, it suffices to show that every finitely generated
object $M$ has a maximal proper subobject and hence a simple quotient
$M \to S$.  Since then, for any finitely generated object $M$ in
$\Aa$, there is a non-zero morphism $M \to S \hookrightarrow E$.  For
an arbitrary object $N$, there exists a non-zero finitely generated
subobject $M \hookrightarrow N$ and so the non-zero morphism $M \to E$
extends to a non-zero morphism $N \to E$.

So, consider the collection $\Mm$ of proper subobjects of $M$, ordered
by inclusion.  Then let $\Ll$ be a totally ordered subset of $\Mm$ and
consider the subobject $\bar{L} := \sum_{L\in \Ll} L$.  If $\bar{L} =
M$, then $M = L$ for some $L \in \Ll$ which contradicts the assumption
that the objects of $\Mm$ are proper subobjects.  Thus $\bar{L}$ is a
proper subobject of $M$ and so is an upper bound of the subset $\Ll$
in $\Mm$.  Applying Zorn's lemma, we conclude that $\Mm$ has a maximal
object as desired. 
\end{proof}
Using some standard arguments we obtain the following corollary.
\begin{corollary}\label{Cor: minimal inj cogen}
  Let $\Ss$ be a set of representatives of the isomorphism class of
  simple objects in $\Aa$.  Then the object
  $\IE(\bigoplus_{S \in \Ss} \IE(S))$ is a minimal injective
  cogenerator of $\Aa$. \qed
\end{corollary} 
\subsection*{Minimal cotilting objects in locally noetherian categories}
Combining the previous two subsections, we obtain the following proposition.
\begin{proposition}\label{prop:minimal}
  Let $\Aa$ be a locally noetherian Grothendieck category.  Then
  \begin{enumerate}
\item Every equivalence class of cotilting objects has a minimal
  representative $C_0$ that is a discrete pure-injective object.
\item The indecomposable direct summands of $C_0$ are in bijection
  with the isomorphism classes of simple objects in $\Gg$.
\end{enumerate}
\end{proposition}
\begin{proof}
  Let $(\Tt, \Ff)$ be a cotilting torsion pair.  Then, by Corollary
  \ref{Cor: loc coh heart plus min cotilt}, the heart $\Gg$ in
  $\Derived{\Aa}$ is locally coherent, so in particular, it is locally
  finitely generated.  By the Corollary \ref{Cor: minimal inj cogen},
  the category $\Gg$ has a minimal injective cogenerator and so by
  Corollary \ref{Cor: loc coh heart plus min cotilt}, there exists a
  minimal cotilting object $C$ such that $\Cogen(C) = \Ff$. Moreover,
  this minimal cotilting objects is discrete since the minimal
  injective cogenerator has no superdecomposable part.
\end{proof}
\begin{lemma}\label{lem:self-orthogonality}
  Let $\Aa$ be a locally noetherian Grothendieck category. Let $E$ be
  a discrete pure-injective object in $\Aa$ with $\id(E)\leq 1$.
  \begin{enumerate}
  \item[(1)] The class $\lperpe{E}$ is closed under products.
  \item[(2)] The following are equivalent:
  \begin{enumerate}
  \item[(a)] $\Ext^1(E,E)=0$.
  \item[(b)] $\Ext^1(E',E'')=0$ for all indecomposable summands
    $E',\,E''$ of $E$.
  \item[(c)] $E$ is self-orthogonal, that is, $\Ext^1(E^I,E)=0$ for
    each set $I$.
  \end{enumerate}
  \end{enumerate}
\end{lemma}
\begin{proof}
  (1) Since $\id(E)\leq 1$, the class $\Ff=\lperpe{E}\cap\fp(\Aa)$ is
  closed under subobjects and extensions, and $\lperpe{E}$ is closed
  under direct limits by
  Theorem~\ref{thm:ext-direct-limit-pure-inj}. As
  in~\cite[Prop.~1.8]{buan:krause:2003} then
  $\lperpe{E}=\varinjlim{\Ff}$ is a torsionfree class, and thus closed
  under products.

  (2) As in~\cite[Cor.~2.3]{buan:krause:2003}.
\end{proof}
\section{Weighted noncommutative regular projective curves}
\label{sec:sheaves-and-modules}
We define the class of weighted noncommutative regular projective
curves by the axioms (NC~1) to (NC~5) below. For details we refer
to~\cite{lenzing:reiten:2006,kussin:2016,angeleri:kussin:2017}. The
content of this background section contains some overlap with
\cite[Sec.~2]{angeleri:kussin:2017} since the settings are the same.
We recall part of the material here for the convenience of the
reader. At the end of this section we exhibit a very useful
  correspondence between cotilting sheaves and tilting sheaves of
  finite type, cf.\ Theorem~\ref{thm:tilting-cotilting-duality}.
\subsection*{The axioms}
A noncommutative curve $\XX$ is given by a category $\Hh$ which is
regarded as the category $\coh\XX$ of \emph{coherent sheaves} over
$\XX$. Formally it behaves like a category of coherent sheaves over a
(commutative) regular projective curve over a field $k$ (we refer
to~\cite{kussin:2016}):
\begin{enumerate}
\item[(NC~1)] $\Hh$ is small, connected, abelian and noetherian.
\item[(NC~2)] $\Hh$ is a $k$-linear category with Hom- and Ext-spaces
  of finite $k$-dimension.
\item[(NC~3)] Serre duality holds in $\Hh$: For all objects
  $X,\,Y\in\Hh$ we have a natural isomorphism
  $$\Ext^1_{\Hh}(X,Y) = \D\Hom_{\Hh}(Y,\tau X)$$ with
  $\D=\Hom_k (-,k)$ and with $\tau\colon\Hh\ra\Hh$ an autoequivalence,
  called Auslander-Reiten translation. (It follows that $\Hh$ is a
  hereditary category without non-zero projective or injective
  object.)
\item[(NC~4)] There exist objects in $\Hh$ of infinite length.
\end{enumerate}
Let $\Hh_0$ denote the class of finite length objects and
$\Hh_+=\vect{\XX}$ the class of a torsionfree objects, also called
vector bundles. Decomposing $\Hh_0$ in its connected components we
have $$\Hh_0=\coprod_{x\in\XX}\Uu_x,$$ where $\XX$ is an index set
(explaining the terminology $\Hh=\coh\XX$) and every $\Uu_x$ is a
connected uniserial length category, a so-called tube. We additionally
assume that $\Hh$ has the following condition.
\begin{enumerate}
\item[(NC~5)] $\XX$ consists of infinitely many points. 
\end{enumerate}
Then $\XX$ (or $\Hh$) is called a \emph{weighted} \emph{noncommutative
  regular projective curve} over $k$. The following statement is
shown in~\cite{kussin:2016}.
\begin{proposition}
  There are (up to isomorphism) only finitely many simple objects in
  $\Uu_x$, for all $x$, and for almost all $x$ there is even only
  one.\qed
\end{proposition}
By $p(x)$ we denote the rank of the tube $\Uu_x$, which is the number
of simple objects in $\Uu_x$ (up to isomorphism). The numbers $p(x)$
with $p(x)>1$ are called the \emph{weights}. The tubes $\Uu_x$ of rank
$1$ are called homogeneous, those finitely many of rank $>1$
non-homogeneous or exceptional. If $S_x$ is a simple object in
$\Uu_x$, then all simple objects (up to isomorphism) in $\Uu_x$ are
given by the Auslander-Reiten orbit
$\tau S_x,\tau^2 S_x,\dots,\tau^{p(x)}S_x=S_x$. \medskip

\emph{In the following, if not otherwise specified, let $\Hh=\coh\XX$
  be a weighted noncommutative regular projective curve.}
\subsection*{The category of quasicoherent sheaves} 
In our focus will be a larger category, the Grothendieck category
$\QHh$. It is obtained from $\Hh$ as the category
$\Lex(\Hh^{\op},\Ab)$, cf.\ \cite[II.~Thm.~1]{gabriel:1962}. We write
$\QHh=\Qcoh(\XX)$ and call the objects quasicoherent sheaves. It is
also of the form $\Qcoh(\Aa)$, the category of quasicoherent modules
over a certain hereditary order $\Aa$; we refer
to~\cite[Thm.~7.11]{kussin:2016}.

The category $\QHh$ is hereditary abelian, and a locally noetherian
Grothendieck category; every object in $\QHh$ is a direct limit of
objects in $\Hh$.  The full abelian
subcategory $\Hh$ consists of the coherent (= finitely presented =
noetherian) objects in $\QHh$, we also write $\Hh=\fp(\QHh)$. Every
indecomposable coherent sheaf has a local endomorphism ring, and $\Hh$
is a Krull-Schmidt category.

\subsection*{Pr\"ufer and adic sheaves}
If $S$ is a simple sheaf, then we denote by $S[n]$ the (unique)
indecomposable sheaf of length $n$ with socle $S$. The inclusions
$S[n]\ra S[n+1]$ $(n\geq 1$) form a direct system, the \emph{ray}
starting in $S$, and their direct union is the \emph{Pr\"ufer sheaf}
$S[\infty]=\varinjlim S[n]$ with respect to
$S$. Cf.~\cite{ringel:1975}.

We denote by $S[-n]$ the (unique) indecomposable sheaf of length $n$
with top $S$. The epimorphisms $S[-n-1]\ra S[-n]$ ($n\geq 1$) form an
inverse system, the \emph{coray} ending in $S$. We write
$S[-\infty]=\varprojlim S[-n]$ for the inverse limit and call it the
\emph{adic sheaf} with respect to $S$. It will be shown in
Lemma~\ref{lem:reduced-pi-are-adics} that $S[-\infty]$ is
indecomposable.
\subsection*{Rank.\ Line bundles} 
Let $\Hh/\Hh_0$ be the quotient category of $\Hh$ modulo the Serre
category of sheaves of finite length, let $\pi\colon\Hh\ra\Hh/\Hh_0$
the quotient functor, which is exact. The \emph{function field} of
$\Hh$ (or of $\XX$) is the up to isomorphism unique skew field
$k(\Hh)$ such that $\Hh/\Hh_0\cong\mod(k(\Hh))$. The
$k(\Hh)$-dimension on $\Hh/\Hh_0$ induces the \emph{rank} of objects
in $\Hh$, which induces a linear form
$\rk\colon\Knull(\Hh)\ra\ZZ$. The objects in $\Hh_0$ are just the
objects of rank zero, every non-zero vector bundle has a positive
rank. The vector bundles of rank one are called \emph{line
  bundles}. For every line bundle $L'$ the endomorphism ring
$\End(L')$ is a skew field. Every vector bundle has a line bundle
filtration, cf.~\cite[Prop.~1.6]{lenzing:reiten:2006}. There exists a
line bundle $L$, called structure sheaf, having certain additional
properties (we refer to~\cite[8.1+Sec.~13]{kussin:2016}).
\subsection*{The sheaf of rational functions}
The \emph{sheaf} $\Kk$ \emph{of rational functions} is the injective
envelope of any line bundle $L$ in the category $\QHh$; this does not
depend on the chosen line bundle. It is torsionfree
by~\cite[Lem.~14]{kussin:2000}, and it is a generic sheaf in the sense
of~\cite{lenzing:1997}; its endomorphism ring is the function field,
$\End_{\QHh}(\Kk)\cong\End_{\Hh/\Hh_0}(\pi L)\cong k(\Hh)$, where
$\pi\colon\Hh\ra\Hh/\Hh_0$ is the quotient functor.
\subsection*{Orbifold Euler characteristic and representation type}
Let $\Hh$ be a weighted noncommutative regular projective curve with
structure sheaf $L$ and $\ovp$ the least common multiple of the
weights. Let $s=s(\Hh)$ be the square root of the dimension of the
function field $k(\Hh)$ over its centre (called the (global)
skewness). We have the \emph{average Euler form}
$\DLF{E}{F}=\sum_{j=0}^{\ovp-1}\LF{\tau^j E}{F}$, and then the
normalized \emph{orbifold Euler characteristic} of $\Hh$ is defined by
$\chi'_{orb}(\XX)=\frac{1}{s^2\ovp^2}\DLF{L}{L}$.\medskip

The orbifold Euler characteristic determines the representation type
of the category $\Hh=\coh\XX$.
\begin{itemize}
\item  $\XX$ is domestic: $\chi'_{orb}(\XX)>0$
\item $\XX$ is elliptic: $\chi'_{orb}(\XX)=0$, and $\XX$ non-weighted
  ($\ovp=1$)
\item $\XX$ is tubular: $\chi'_{orb}(\XX)=0$, and $\XX$ properly
  weighted ($\ovp>1$)
\item $\XX$ is wild: $\chi'_{orb}(\XX)<0$.
\end{itemize}
\subsection*{Degree and slope}
With the structure sheaf $L$ we define the \emph{degree} function
$\deg\colon\Knull(\Hh)\ra\ZZ$ by
\begin{equation}
  \label{eq:deg-define}
  \deg(F)=\frac{1}{\kappa\varepsilon}\DLF{L}{F}-\frac{1}{\kappa\varepsilon}\DLF{L}{L}\rk(F).
\end{equation}
Here, $\kappa=\dim_k\End(L)$, and $\varepsilon$ is the positive
integer such that the resulting linear form $\Knull(\Hh)\ra\ZZ$
becomes surjective. We have $\deg(L)=0$, and $\deg$ is positive and
$\tau$-invariant on sheaves of finite length.

The \emph{slope} of a non-zero coherent sheaf $F$ is defined as
$\mu(F)=\deg(F)/\rk(F)\in\widehat{\QQ}=\QQ\cup\{\infty\}$, and $F$ is
called \emph{stable} (\emph{semistable}, resp.) if for every non-zero
proper subsheaf $F'$ of $F$ we have $\mu(F')<\mu(F)$ (resp.\
$\mu(F')\leq\mu(F)$).
\subsection*{Torsion, torsionfree, divisible and reduced sheaves}
\begin{numb}
  The class of \emph{torsionfree} (quasicoherent) sheaves is given by
  $\Ff=\rperpo{{\Hh_0}}$. We have $\vect\XX=\Ff\cap\Hh$. The class of
  \emph{torsion} sheaves is given as the direct limit closure
  $\Tt=\vec{\Hh_0}=\varinjlim{\Hh_0}$. The pair $(\Tt, \Ff)$ is a
  torsion pair in $\QHh$. Every $E\in\QHh$ has a largest subsheaf from
  $\Tt$, the torsion subsheaf $tE$. The canonical sequence $0\ra tE\ra
  E\ra E/tE\ra 0$ is pure-exact, and $E/tE$ is torsionfree.

  Let $V\subseteq\XX$ be a subset. The class of $V$-\emph{divisible} sheaves
  is $\Dd_V=\rperpe{{\bigl(\coprod_{x\in V}\Uu_x\bigr)}}$. It is
  closed under direct summands, set-indexed direct sums, extensions
  and epimorphic images. If $V = \{x\}$, then we will refer to
  $V$-divisible sheaves as \emph{$x$-divisible}. In case $V=\XX$ we
  just say \emph{divisible}.

  The class $\Dd=\Dd_{\XX}$ of divisible sheaves is a torsion class,
  and the corresponding torsion pair $(\Dd,\Rr)$ in $\QHh$ splits. The
  sheaves in $\Rr$ are called \emph{reduced}. By
  \cite[Lem.~3.3]{angeleri:kussin:2017} we have $\Dd=\Inj(\QHh)$, the
  class of injective sheaves. Moreover, the indecomposable injective
  sheaves are (up to isomorphism) the sheaf $\Kk$ of rational
  functions and the Pr\"ufer sheaves $S[\infty]$ ($S\in\Hh$ simple).
\end{numb}
\subsection*{A tilting-cotilting correspondence}
\begin{proposition}\label{prop:tf-gen-equals-resolve}
  Let $\QHh=\Qcoh\XX$ be a weighted noncommutatuve regular projective
  curve and $\Hh=\coh\XX$. A class $\Ff\subseteq\Hh$ is torsionfree
  and generating if and only if it is resolving (in the sense
  of~\cite[Def.~4.2]{angeleri:kussin:2017}).
  \end{proposition}
\begin{proof}
  This is~\cite[Cor.~4.17]{angeleri:kussin:2017}.
\end{proof}
We recall that cotilting sheaves $C,\,C'$ are equivalent if
$\Cogen(C)=\Cogen(C')$. Similarly, tilting sheaves $T,\,T'$ are
equivalent if $\Gen(T)=\Gen(T')$. By slight abuse of notation we
denote the equivalence classes in both cases in the same way as $[C]$
and $[T]$, respectively.
\begin{theorem}\label{thm:tilting-cotilting-duality}
  If $T$ is a tilting sheaf of finite type there is a cotilting sheaf
  $C$ such that $\lperpe{C}\cap\Hh=\lperpe{(\rperpe{T})}\cap\Hh$, and
  conversely. The assignments $\Gamma\colon [T]\mapsto [C]$ and
  $\Theta\colon [C]\mapsto [T]$ induce mutually inverse bijections
  between the sets of
  \begin{itemize}
  \item equivalence classes $[T]$ of tilting sheaves $T$ of finite
    type, and of
  \item equivalence classes
  $[C]$ of cotilting sheaves $C$.
  \end{itemize}
\end{theorem}
\begin{proof}
  Follows from the preceding proposition by invoking
  Theorem~\ref{thm:cotilting-is-of-finite-type}
  and~\cite[Thm.~4.14]{angeleri:kussin:2017}.
\end{proof}
For simplicity, we will even just write $\Gamma (T)=C$, without
  brackets.

\section{Pure-injective sheaves of  slope infinity}\label{sec:pure-inj-inf-slope}
Let $\QHh=\Qcoh\XX$ be a weighted noncommutative regular projective
curve over a field
$k$. Let $$\Mm(\infty)=\lperpo{\vect\XX}=\rperpe{(\vect\XX)}.$$ The
sheaves in $\Mm(\infty)$ are said to have slope $\infty$. Examples are
the torsion sheaves, but also the generic and the adic sheaves, which
are torsionfree. Moreover:

\begin{lemma}\label{lem:V-divisibles-have-slope-infty}
  For every non-empty subset $V$ of $\XX$ the $V$-divisible sheaves
  have slope $\infty$: $\Dd_V\subseteq\Mm(\infty)$.
\end{lemma}
\begin{proof}
  Let $E$ be $x$-divisible for some point $x$. If there is a non-zero
  morphism from $E$ to a vector bundle, then there is also an
  epimorphism to a line bundle $L'$. Since there is an epimorphism
  from $L'$ to a simple object $S_x$ concentrated in $x$, we get with
  Serre duality a contradiction to $x$-divisibility.
\end{proof}
\begin{proposition}
  Let $E\in\QHh$ be pure-injective, torsionfree and of slope
  $\infty$. Then $E$ is rigid.
\end{proposition}
\begin{proof}
  Since $E$ is torsionfree, we have $E=\varinjlim E_i$ for a directed
  system of vector bundles $(E_i)_{i\in I}$. The claim then follows
  from Theorem~\ref{thm:ext-direct-limit-pure-inj}.
\end{proof}
\begin{corollary}
  Let $E,\,F\in\{\Kk,\, S[-\infty]\mid S\ \text{simple}\}$. Then
  $\Ext^1(E,F)=0$.
\end{corollary}
\begin{proof}
  By the proposition, $E\oplus F$ is rigid.
\end{proof}
\subsection*{Definability}
Let $\Aa$ be a locally coherent Grothendieck category with
$\Aa_0=\fp(\Aa)$. A full subcategory $\Cc$ of $\Aa$ is called \emph{definable} if
it is closed under products, direct limits and pure subobjects. 
\begin{proposition}\label{prop:slope-infty-definable}
  $\Mm(\infty)$ is a definable subcategory of $\Qcoh\XX$.
\end{proposition}
\begin{proof}
  Let $X_i$ be a family of objects in $\Aa$ and
  $E\in\vect\XX$. By~\cite[Cor.~A.2]{coupek:stovicek:2019} we have
  $\Ext^1(E,\prod_{i}X_i)=0$ if and only if $\Ext^1(E,X_i)=0$ for all
  $i$. Hence $\Mm(\infty)$ is closed under products.

  Assume that the $X_i$ form a directed set of objects. Then
  $\Hom(\varinjlim X_i,E)\cong\varprojlim\Hom(X_i,E)$, and thus
  $\Mm(\infty)$ is also closed under direct limits.

  By applying $\Hom(E,-)$ to a pure-exact sequence
  $0\ra X\ra Y\ra Z\ra 0$ with $Y\in\rperpe{(\vect\XX)}$, the
  resulting long exact sequence shows $X\in\rperpe{(\vect\XX)}$, and
  thus $\Mm(\infty)$ is closed under pure subobjects.
\end{proof}

\subsection*{Pure-injectives of slope $\infty$}
We wish to determine the indecomposable pure-injective objects in the
class $\Mm:= \Mm(\infty)$ of objects of slope $\infty$.  In order to
do this, we consider the equivalent category $\Mm$ that occurs as a
subcategory of $\Derived{\QHh}$.

In the derived category $\Derived{\QHh}$ and the compact objects are
given by
$\bDerived{\Hh} = \add\left( \bigvee _{n\in \mathbb{Z}}\Hh[n]\right)$
(see the introduction for an explanation of this notation).  Given
this description of $\bDerived{\Hh}$, we may partition
$\bigvee _{n\in \mathbb{Z}}\Hh[n]$ into three
parts:\[ \bp := \left(\bigvee _{n <0}\Hh[n]\right) \vee
  \vect\mathbb{X} \hspace{5mm} \bt := \Hh_0 \hspace{5mm} \bq :=
  \left(\bigvee _{n >0}\Hh[n]\right).
\]

We also consider the HRS-tilted t-structure $(\Uu_\Dd, \Vv_\Rr)$ of
the split torsion pair $(\Dd, \Rr)$ in $\Derived{\QHh}$ where we take
the class $\Dd := \{M\in \QHh \mid \Hom(M, \Hh_0)=0\}$ of divisible
objects and the class $\Rr := \rperpo{\Dd}$ of reduced objects (see
Section \ref{sec:sheaves-and-modules}).  By
Proposition~\ref{prop:class-of-pure-injective-complexes} applied to
$\Derived{\QHh}$ the pure-injective objects in $\Derived{\QHh}$ are
exactly those in the class $\Prod(\bDerived{\Hh})$ and so we will use
the partition $(\bp, \bt, \bq)$, to find the indecomposable objects in
$\Prod(\bDerived{\Hh}) \cap \Mm$.
\begin{lemma}\label{lem:products-fin-length}
  Let $\{X_i\}_{i\in I}$ be a set of objects in $\Hh_0$.  Then in
  $\Tt:=\Derived{\QHh}$ we have
  $\prod^\Tt_{i\in I} X_i \cong ( \prod^{\QHh}_{i\in I} X_i)$, where
  the product on the left is taken in $\Derived{\QHh}$ and the product
  on the right is taken in $\QHh$.
\end{lemma}
\begin{proof}
  Note that the class $\Ff$ of torsionfree objects in $\QHh$ is a
  torsionfree class that contains a system of generators. Moreover, by
  definition, we have that $\Ff = \rperpo{{\Hh_0}}$.  By applying
  (generalised) Serre duality, we have that
  $\Hh_0\subseteq\Ff^{\perp_1}$.  Thus, by
  \cite[Prop.~2.12]{coupek:stovicek:2019}, we have the desired result.
\end{proof}
\noindent The following lemma is a derived version
of~\cite[2.2]{ringel:2004}. In the proof, we will make use of the
following setup several times; we will refer to it as Setup (*).  For
every pure-injective object $X$ in $\Derived{\QHh}$, then we have the
following morphisms:
  \[ X \overset{\left(\begin{smallmatrix} f_\bp \\ f_\bt \\
          f_\bq \end{smallmatrix} \right)}{\longrightarrow} X_\bp
    \oplus X_\bt \oplus X_\bq \overset{\left(\begin{smallmatrix} g_\bp
          & g_\bt & g_\bq \end{smallmatrix} \right)}{\longrightarrow}
    X \] such that $g_\bp f_\bp + g_\bt f_\bt + g_\bq f_\bq = 1$ where
  $X_\bp$ is a product of objects in $\bp$, $X_\bt$ is a product of
  objects in $\bt$ and $X_\bq$ is a product of objects in $\bq$.  All
  products and $\Prod(-)$ are taken in $\Derived{\QHh}$ unless
  otherwise stated.
  \begin{lemma}\label{lem:trisection}
    We have
\begin{enumerate}
\item \textbf{(i)} $\Prod(\bp) \cap \Mm = 0$, \textbf{(ii)}
  $\Prod(\bt) \subseteq \Mm$, \textbf{(iii)}
  $\Prod(\bq) \cap \Mm = \Dd$.
\item The class $\Prod(\bt)$ is the class of pure-injective objects in
  $\Rr \cap \Mm$.
\item The following are equivalent for $X \in \QHh$.
  \begin{enumerate}
  \item $X = X' \oplus X''$ where $X' \in \Prod(\bt)$ and
    $X'' \in \Dd$.
\item $X$ is pure-injective and belongs to $\Mm$.
\end{enumerate}
\end{enumerate}
\end{lemma}
\begin{proof}
  It follows from
  Proposition~\ref{prop:summands-of-products} that the classes
  $\Prod(\bp)$, $\Prod(\bt)$ and $\Prod(\bq)$ have pairwise zero
  intersections.
  
  (1)(i) Let $M\in \Mm$, then
  $\Hom_{\Derived{\QHh}}(M, \bigvee _{n <0}\Hh[n]) = 0$ and
  $\Hom_{\QHh}(M, \vect\mathbb{X}) = 0$.  That is, we have
  $\Hom_{\Derived{\QHh}}(M, \bp) = 0$ and so the first claim follows.

  (ii) Next, note that $\Mm$ is closed under products in $\QHh$ by
  Proposition~\ref{prop:slope-infty-definable}. Then, as
  $\bt \subseteq \Mm$, it follows from Lemma
  \ref{lem:products-fin-length} that $\Prod(\bt)\subseteq \Mm$.

  (iii) For the third claim, let $X\in \Dd$.  As $\Dd$ consists of
  pure-injective objects in $\QHh$, it follows from Lemma
  \ref{lem:pure-in-der-cat} that $X$ is a pure-injective object in
  $\Derived{\QHh}$ and so we are in Setup (*) above.  We have that
  $\Hom_{\Derived{\QHh}}(\Dd, \Prod(\bt)) = 0$ because
  $\Hom_{\QHh}(\Dd, \Hh_0) = 0$, so $f_\bt=0$.  Similarly, we have
  that $f_\bp =0$ because $\Hom_{\QHh}(\Dd, \vect\mathbb{X}) =0$ and
  $\Hom_{\Derived{\QHh}}(\Dd, \bigvee _{n <0}\Hh[n])=0$.  So
  $X \in \Prod(\bq)$.  Moreover, we have that $\Dd\subseteq\Mm$ so we
  have shown that $\Dd \subseteq \Prod(\bq)\cap\Mm$ in
  $\Derived{\QHh}$.

  We wish to show that $\Prod(\bq) \cap \Mm \subseteq\Dd$; in fact we
  will show that $\Prod(\bq) \cap \QHh \subseteq\Dd$.  Let
  $Y \in \Prod(\bq) \cap \QHh$.  We will show that
  $\Hom_{\Derived{\QHh}}(S_x[-1], Y) \cong \Hom_{\Derived{\QHh}}(S_x,
  Y[1]) \cong \Ext^1_{\QHh}(S_x, Y) = 0$ for all simple
  $S_x \in \Hh_0$ i.e.~$Y\in \Dd$.  Suppose, for a contradiction, that
  there is a non-zero morphism $f \colon S_x[-1] \to Y$.  Since
  $Y \in \Prod(\bq)$, it follows that there exists a non-zero morphism
  $g \colon Y \to Q$ with $Q\in\bq$ indecomposable such that
  $gf \neq 0$.  By definition, we have $Q \cong X[i]$ for some
  $X\in \Hh$ and $i>0$.  But then we have
  $0 \neq gf \in \Hom_{\Derived{\QHh}}(S_x[-1], X[i]) \cong
  \Ext^{i+1}_{\QHh}(S_x, Y)$, which is a contradiction.  Therefore we
  must have that $Y \in \Dd$.

  (2) We have already seen that $\Prod(\bt)$ consists of
  pure-injective objects and also $\Prod(\bt) \subseteq \Mm$.
  Since $(\Dd, \Rr)$ is a split torsion pair in $\QHh$, we
  can write any $X\in \Prod(\bt)$ as $X \cong X_\Dd\oplus X_\Rr$ where
  $X_\Dd\in \Dd$ and $X_\Rr\in \Rr$.  But then
  $X_\Dd \in \Prod(\bq) \cap \Prod(\bt) = 0$.  So
  $X\cong X_\Rr \in \Rr$.  We have shown that
  $\Prod(\bt) \subseteq \Mm\cap\Rr$.

  For the reverse inclusion, let $X\in \Mm \cap \Rr$ be
  pure-injective.  Again, we are in Setup (*) above.  In the proof of
  (1)(i) we saw that $\Hom_{\Derived{\QHh}}(X, \bp) = 0$, therefore
  $f_\bp =0$.  We will show that $g_\bq =0$.  Note that
  $X_\bq \cong \bigoplus_{i\geq 0} Q_i[i]$ where
  $Q_i \cong H^{-i}(X_\bq) \in\QHh$ for each $i\geq0$ (see, for
  example, \cite[Sec.~1.6]{krause:2007}).  Then
  $\Hom_{\Derived{\QHh}}(X_\bq, X) \cong \prod_{i\geq 0}
  \Hom_{\Derived{\QHh}}(Q_i[i], X) \cong \Hom_{\Derived{\QHh}}(Q_0[0],
  X)$.  In the proof of (1)(iii) we saw that
  $\Prod(\bq) \cap \QHh[0] \subseteq\Dd[0]$, and so we have
  $\Hom_{\Derived{\QHh}}(Q_0, X) = 0$ because $X \in \Rr$ and
  $Q_0 \in \Dd$.  We therefore have that $X \in \Prod(\bt)$.

  (3) (a) $\Rightarrow$ (b) This implication is clear since
  $\Prod(\bt)$ and $\Dd$ both consist of pure-injective objects.

  (b) $\Rightarrow$ (a) Suppose $X \in \Mm$ is pure-injective.  Using
  that split torsion pair $(\Dd, \Rr)$, we may decompose $X$ as
  $X_\Dd\oplus X_\Rr$.  Then $X_\Dd \in \Dd$ and
  $X_\Rr \in \Rr \cap \Mm$.  Since $X_\Rr$ is pure-injective, we have
  $X_\Rr \in \Prod(\bt)$ by part (2).
\end{proof}
\noindent We obtain in particular:
\begin{lemma}\label{lem:reduced-sheaves-of-slope-infty}
  In $\QHh$ we have that $\Prod(\Hh_0)$ is the class of pure-injective
  sheaves in $\Rr\cap\Mm(\infty)$. 
\end{lemma}
\begin{proof}
  By Lemma \ref{lem:trisection}, we have that $\Prod(\Hh_0[0])$
  coincides with the class of pure-injective objects in
  $\Rr[0]\cap\Mm[0]$ in $\Derived{\QHh}$. The result then follows from
  Lemma \ref{lem:products-fin-length}.
\end{proof}
\begin{lemma}\label{lem:reduced-not-finite-length-implies-torsionfree}
  Let $M\in\QHh$ be reduced and having no non-zero direct summand of
  finite length. Then $M$ is torsionfree.
\end{lemma}
\begin{proof}
  The canonical sequence $0\ra tM\ra M\ra M/tM\ra 0$, where $tM$ is
  the largest torsion subsheaf of $M$, is pure-exact. Since finite
  length sheaves are pure-injective it follows that $tM$ as well has
  no non-zero direct summand of finite length, and thus $tM$ is a
  coproduct of Pr\"ufer sheaves (cf.\
  \cite[Cor.~3.7(2)]{angeleri:kussin:2017}). Since $M$ is reduced,
  this coproduct must be empty, that is, $tM=0$.
\end{proof}
\begin{lemma}\label{lem:reduced-pi-are-adics}
  Let $0\neq M\in\QHh$ be torsionfree and lying in $\Prod(\Uu_x)$ for
  some $x\in\XX$. Then $M$ has an indecomposable direct summand
  isomorphic to the adic $S_x[-\infty]$ for some simple $S_x\in\Uu_x$.
\end{lemma}
\begin{proof}
  (1) If $\QHh=\Qcoh(\Aa)$ with a hereditary order $\Aa$, let $R_x$ be
  the endomorphism ring of $\Aa$ considered as an object in the
  quotient category
  $\QHh_x=\QHh/\varinjlim\bigl(\coprod_{y\neq x}\Uu_y\bigr)$. Since
  $\Uu_x\subseteq\rperp{\{\Uu_y\mid y\neq x\}}\simeq\Mod(R_x)$, we
  consider $M$ as an $R_x$-module. Moreover, $\varinjlim{\Uu_x}$ is
  the class of torsion modules over $R_x$, and the above
  right-perpendicular category is closed in $\QHh$ under limits (in
  particular: products) and direct limits (cf.\
  \cite{krause:1997}). Thus an $R_x$-module is (pure-) injective if
  and only if it is (pure-) injective in $\QHh$. In particular, $M$ is
  a reduced, torsionfree and pure-injective $R_x$-module. Since $M$ is
  a direct summand of a product of modules in $\Uu_x$, which are
  complete, it is also complete.

  (2) We first treat the special case where $p(x)=1$. By completeness,
  $M$ is a $V_x$-module, where $V_x$ is a complete discrete valuation
  domain with $\widehat{R}_x\cong\matring_{e(x)}(V_x)$, where
  $\widehat{R}_x=\varprojlim R_x/\rad(R_x)^i$ is the $\rad(R_x)$-adic
  completion of $R_x$; we refer to \cite[Prop.~3.16]{kussin:2016}. The
  class of torsion (resp., finite length) $V_x$-modules coincides with
  the class of torsion (resp., finite length) $R_x$-modules. In
  particular, it follows that $M$ is also reduced, torsionfree and
  pure-injective as a $V_x$-module. Since $M$ is reduced, it has a
  maximal submodule. Each $a\in M\setminus\rad M$ induces a
  monomorphism $f\colon V_x\ra M$ which does not belong to
  $\rad(V_x,M)$. Since (cf.\ \cite[after
  Ex.~11.9]{krylov:tuganbaev:2008})
  $\End_{R_x}(V_x)=\End_{V_x}(V_x)\cong V_x$ is local, we obtain that
  $f$ splits, and thus $V_x$ is a direct summand of $M$. (For a
  similar argument cf.\ \cite[Cor.~11.6]{krylov:tuganbaev:2008}.) By
  construction, the $R_x$-module $V_x=\varprojlim V_x/\rad(V_x)^i$
  corresponds to the adic $S_x[-\infty]$ in $\QHh$, where $S_x$
  (corresponding to the $R_x$-module $V_x/\rad(V_x)$) is the only
  simple in $\Uu_x$.

  (3) Now let $p=p(x)$ be arbitrary. Then we have to replace the
  complete ring $V_x$ by the ring $H=H_p(V_x)$,
  see~\cite[Prop.~13.4]{kussin:2016}. Since $M\in\Prod(\Uu_x)$, and
  with the same arguments as in~(2), $M$ is a complete, torsionfree,
  reduced and pure-injective $H$-module. Since $M$ is reduced, there
  is $a\in M\setminus\rad M$. This induces a monomorphism
  $f\colon H\ra M$ with $f\not\in\rad(H,M)$. Let $e_1,\dots,e_p$ be
  the canonical complete set of primitive, othogonal idempotents of
  $H$. There is some $i$ and a morphism $f_i\colon e_i H\ra M$ with
  $f_i\not\in\rad (e_iH,M)$. Moreover $\End(e_i H)=e_i He_i\cong V_x$
  is local. It follows that $f_i$ is a split monomorphism. Thus
  $e_i H$ is an indecomposable direct summand of $M$, and it
  corresponds to the adic associated with some simple in $\Uu_x$ (cf.\
  also~\cite[4.4]{ringel:1979}).
\end{proof}
\begin{remark}
  The lemma shows in particular that all the adics $S[-\infty]$ are
  indecomposable.
\end{remark}
\begin{theorem}\label{thm:ind-pure-inj-slope-infty}
  The following is a complete list of the indecomposable
  pure-injective sheaves in $\QHh=\Qcoh\XX$ of slope $\infty$:
  \begin{enumerate}
  \item[(1)] The indecomposable sheaves of finite length.
  \item[(2)] The sheaf $\Kk$ of rational functions, the Pr\"ufer and
    the adic sheaves.
  \end{enumerate}
  Moreover, each pure-injective sheaf $E$ of slope $\infty$ is
  discrete, that is, has -- unless zero -- an indecomposable direct
  summand.
\end{theorem}
\begin{proof}
  We assume that $M$ is indecomposable pure-injective of slope
  $\infty$ and not coherent. Since $M$ is indecomposable, $M$ is
  either divisible or reduced. In the first case it is generic or
  Pr\"ufer. Thus we can assume that $M$ is reduced, and we have to
  show that $M$ is an adic. By
  Lemma~\ref{lem:reduced-sheaves-of-slope-infty} we have
  $M\in\Prod(\Hh_0)$. Since $M$ is indecomposable there is $x\in\XX$
  such that even $M\in\Prod(\Uu_x)$ (cf.~\cite[2.3]{ringel:2004}; the
  arguments therein also hold in our setting). Since $M$ is not of
  finite length, it is torsionfree by
  Lemma~\ref{lem:reduced-not-finite-length-implies-torsionfree}. By
  the Lemma~\ref{lem:reduced-pi-are-adics} then $M$ is an adic with
  respect to $\Uu_x$.

  The additional statement follows also from that lemma. Indeed, it is
  sufficient to assume that $E$ is reduced and moreover belonging to
  $\Prod(\Uu_x)$ for some $x$.
\end{proof}
\begin{proposition}
  For every simple $S$ there is a short exact sequence
  $$0\ra \tau S[-\infty]\ra E\ra S[\infty]\ra 0$$ with $E$ a direct
  sum of copies of $\Kk$.
\end{proposition}
\begin{proof}
  Let $p\geq 1$ be the rank of the tube $\Uu_x$ containing $S$. As
  in~\cite{krause:1998} we get by an inverse limit construction (using
  $S[-p]\cong \tau^- S[p]$) a short exact sequence
  $0\ra\tau S[-\infty]\ra\tau S[-\infty]\ra S[p]\ra 0$; for exactness
  of the inverse limit we note that we may form the inverse limit of a
  surjective inverse system in $\QHh_x=\Mod(R_x)$ as in the proof of
  Lemma~\ref{lem:reduced-pi-are-adics}. Then by a direct limit
  construction we get a short exact sequence
  $0\ra\tau S[-\infty]\ra E\ra S[\infty]\ra 0$; it follows as
  in~\cite[Prop.~4]{ringel:1998} that $E$ is torsionfree and
  divisible, hence a direct sum of copies of $\Kk$.
\end{proof}
The proposition shows (cf.\ \cite[Lem.~2.7]{buan:krause:2003}):
\begin{lemma}\label{lem:pruefer-adic-ext-relations}
  Let $x,\,y\in\XX$ and $j$ be an integer. Then
  \begin{enumerate}
  \item[(1)] $\Ext^1 (S_x[\infty],\tau^j S_y[-\infty])\neq 0$ if and
    only if $x=y$.
  \item[(2)] $\Ext^1 (S_x[-\infty],\tau^j S_y[\infty])=0$. \qed
  \end{enumerate}
\end{lemma}

\section{Cotilting sheaves of slope infinity}\label{sec:slope-infinity}
We will classify all cotilting sheaves having slope $\infty$.
\begin{proposition}\label{prop:tilting-cotilting-invariances}
  We have the following.
  \begin{enumerate}
  \item[(1)] Let $C$ be a cotilting sheaf and
    $\Ff_0=\lperpe{C}\cap\Hh$. Then $C$ has slope $\infty$ if and only
    if $\vect\XX\subseteq\Ff_0$.
  \item[(2)] Let $C$ and $C'$ be two equivalent cotilting sheaves. If one of them
    has slope $\infty$, then so has the other.
  \item[(3)] Each cotilting sheaf of slope $\infty$ is large.
  \item[(4)] Let $C$ be a cotilting sheaf and $T$ be a tilting sheaf
    such that $\Gamma(T)=C$, with $\Gamma$ as in
    Theorem~\ref{thm:tilting-cotilting-duality}. Then $C$ has slope
    $\infty$ if and only if $T$ has slope $\infty$.
  \end{enumerate}
\end{proposition}
\begin{proof}
  (1) is clear.
  
  (2) Follows from $\Prod(C)=\Prod(C')$.

  (3) Follows since there is no cotilting sheaf consisting only of
  indecomposable summands of finite length
  (cf.~\cite[Rem.~7.7]{angeleri:kussin:2017}).

  (4) Follows from Theorem~\ref{thm:tilting-cotilting-duality} and~(1)
  (and its analogue for tilting objects).
\end{proof}
\subsection*{Rigidity}
The following basic splitting property will be crucial for our
treatment of cotilting sheaves.
\begin{theorem}[{\cite[Thm.~3.8]{angeleri:kussin:2017}}]\label{thm:torsion-splitting}
  Let $E\in\QHh$ be a rigid sheaf, that is, $\Ext^1(E,E)=0$ holds.
\begin{enumerate}
\item[(1)]The torsion subsheaf $tE$ is a direct sum of Pr\"ufer
  sheaves and exceptional sheaves of finite length. Accordingly it is
  pure-injective.  \smallskip
  \item[(2)] The canonical exact sequence $0\ra tE\ra E\ra E/tE\ra 0$
    splits. \qed
  \end{enumerate}
\end{theorem}
Given a rigid sheaf $E\in\QHh$, we will use the notation
$$E=E_+\oplus E_0$$ where $E_0=tE$ denotes the \emph{torsion part} and
$E_+\cong E/tE$ denotes the \emph{torsionfree part} of $E$.  We will
say that $E$ has a \emph{large torsion part} if there is no coherent
sheaf $F$ such that $\Add(tE)=\Add(F)$.
\begin{corollary}
  Let $E\in\QHh$ be rigid and indecomposable. Then $E$ is either
  torsion or torsionfree. \qed
\end{corollary}
\begin{proposition}
  For any sheaf $E$, if the Pr\"ufer sheaf $S[\infty]$ belongs to
  $\Prod(E)$, then it is a direct summand of $E$.
\end{proposition}
\begin{proof}
  Suppose $S[\infty]$ is a direct summand of $E^I$ and let
  $s \colon S[\infty] \to E^I$ be the corresponding split
  monomorphism.  Also let $m \colon S \to S[\infty]$ be an essential
  monomorphism, which exists since $S[\infty]$ is the injective
  envelope of the simple object $S$.  Then $s\circ m$ is non-zero and
  hence $\pi \circ s \circ m$ is a non-zero monomorphism for some
  projection $\pi \colon N^I \to N$.  As $m$ is an essential
  monomorphism, it follows that $\pi\circ s$ is a monomorphism.  As
  $S[\infty]$ is injective, the proposition follows.
\end{proof}
\begin{corollary}\label{cor:pruefer-summands-of-products}
  Let $C\in\QHh$ be cotilting with torsionfree class $\Ff=\lperpe{C}$.
  If $S[\infty]\in\Ff$, then it is a direct summand of $C$.
\end{corollary}
\begin{proof}
  Since $S[\infty]$ is injective, we have
  $S[\infty]\in\Ff\cap\rperpe{\Ff}=\Prod(C)$.
\end{proof}
\subsection*{Maximal self-orthogonality w.r.t.\ tubes}
Let $\Uu$ be a tube. As in~\cite{buan:krause:2003} we say that a
pure-injective object $M$ \emph{belongs to $\Uu$} if every
indecomposable direct summand of $M$ is of the form $S[n]$ with
$S\in\Uu$ simple and $n\in\NN\cup\{\pm\infty\}$. The subcategory
formed by all such objects is denoted by $\overline{\Uu}$. The
\emph{$\Uu$-component} $M_{\Uu}$ of $M$ is defined to be a maximal
direct summand of $M$ belonging to $\overline{\Uu}$. If
$\Uu = \Uu_x$, we will also use the notation $M_x$. The
$\Uu$-component is unique up to isomorphism. In this context it is
useful to recall that each indecomposable pure-injective object has a
local endomorphism ring and we have the exchange property for such
objects $U$, cf.\ \cite[Thm.~E.1.53.]{prest:2009}; for instance, of
$U$ is a direct summand of a direct sum $M\oplus N$, then it is a
direct summand of $M$ or of $N$. Moreover, the $\Uu$-component of
  a cotilting object $M$ is said to be of \emph{Pr\"ufer type} (resp.,
  \emph{adic type}) if it has a Pr\"ufer (resp., adic) summand; it
  will turn out that each component is either of Pr\"ufer or of adic
  type, and not both.

The following lemma is shown as
in~\cite[Prop.~3.3]{buan:krause:2003}.
\begin{lemma}\label{lem:maximality-in-tubes}
  Let $C$ be cotilting of slope $\infty$ and $\Uu$ a tube. Then the
  $\Uu$-component $C_{\Uu}$ is maximal self-orthogonal with respect to
  all objects in $\overline{\Uu}$. \qed
\end{lemma}
\subsection*{Branches}
\begin{numb}\label{nr:wings} {\bf Branch sheaves.} 
  In this section we consider certain coherent sheaves, called branch
  sheaves, which turn out to be typical coherent summands of large
  cotilting sheaves.  Let $\Uu_x$ be a tube of rank $p>1$.  The
  exceptional (i.e.~indecomposable and rigid) sheaves $E$ in $\Uu_x$
  are exactly those of length $\leq p-1$ and so there are only
  finitely many of them. The collection $\Ww$ of subquotients of an
  exceptional sheaf $E$ is called the \emph{wing rooted in $E$} and
  $E$ is called the \emph{root of $\Ww$}.  The set of simple sheaves
  in $\Ww$ is called the \emph{basis} of $\Ww$.  The basis of any wing
  is of the form $S,\,\tau^- S,\dots,\tau^{-(r-1)}S$ for a simple
  sheaf $S$ where $r$ is the length of the root $E$.  We say that
  another wing $\Ww'$ is \emph{not adjacent} to $\Ww$ if their bases
  are disjoint and neither $\tau S$ nor $\tau^{-r}S$ is in $\Ww'$ (we
  say that the two wings $\Ww$ and $\Ww'$ are
  \emph{non-adjacent})~\cite[Ch.~3]{lenzing:meltzer:1996}.\medskip 

  The full subcategory $\add\Ww$ is equivalent to the category of
  modules over the path algebra of a linearly ordered Dynkin quiver of
  type $\AAA$, cf.~\cite[Ch.~3]{lenzing:meltzer:1996}.  We define a
  \emph{connected branch $B$} in $\Ww$ in the following way: $B$ has
  exactly $r$ nonisomorphic indecomposable summands $B_1, \dots , B_r$
  such that $B_1 \cong E$ and for every $j$, the wing rooted in $B_j$
  contains exactly $\ell_j$ indecomposable summands of $B$ where
  $\ell_j$ is the length of $B_j$.  Each connected branch in $\Ww$ is
  a tilting object in the subcategory
  $\add{\Ww}$~\cite[p.~205]{ringel:1984}.  \medskip

  A module $B$ in $\Hh_0$ is called a \emph{branch sheaf} if it is a
  multiplicity-free direct sum of connected branches in pairwise
  non-adjacent wings~\cite[Ch.~3]{lenzing:meltzer:1996}.  Any branch
  sheaf $B$ is rigid and decomposes as $B=\bigoplus_{x\in\XX}B_x$.  In
  fact, it is clear from the definition that $B_x = 0$ for all
  $x\in \XX$ corresponding to homogeneous tubes and there are only
  finitely many isomorphism classes of branch sheaves.  \medskip

Given a non-empty subset $V\subseteq\XX$, we also write
$$B=B_{\mathfrak i}\oplus B_{\mathfrak e}$$ where $B_{\mathfrak e}$ is
supported in $\XX\setminus V$ and $B_{\mathfrak i}$ in $V$. In this
case we will say that $B_{\mathfrak e}$ is \emph{exterior} and
$B_{\mathfrak i}$ is \emph{interior} with respect to $V$. We will see
in Theorem \ref{thm:large-cotilting-sheaves-in-general} that a pair
$(B,V)$ determines a cotilting module $C$, in which the exterior part
of $B$ with respect to $V$ determines the adic summands of $C$ and the
interior part of $B$ with respect to $V$ determines the Pr\"ufer
summands of $C$.
\end{numb}
\begin{lemma}\label{lem:exceptional-tube}
  Let $C$ be a cotilting sheaf of slope $\infty$ and $x$ a point of
  weight $p=p(x)\geq 1$. There are two possible cases:
  \begin{enumerate}
  \item[(1)] Exterior ``adic type'' case:
    \begin{enumerate}
    \item The $\Uu_x$-component $C_x$ of $C$ contains no Pr\"ufer
      sheaf. The torsion part of $C_x$ consists of a direct sum
      of $0\leq s\leq p-1$ indecomposable summands of finite length.
      \item The finite length summands are arranged in connected
        branches in pairwise non-adjacent wings; let $\Ww$ denote the
        union of these wings in $\Uu_x$.
      \item The $p-s$ adic sheaves $S_x[-\infty]$ such that
        $\Hom(S_x[-\infty],\tau\Ww)=0$ are (torsionfree) direct
        summands of $C$.\smallskip
      \end{enumerate}
\item[(2)] Interior ``Pr\"ufer type'' case:
  \begin{enumerate}
  \item The $\Uu_x$-component $C_x$ of $C$ consists of a direct sum of
    $1\leq s\leq p$ Pr\"ufer sheaves, and precisely $p-s$
    indecomposable summands of finite length.
    \item The finite length summands belong to
    wings of the following form: if $S[\infty]$, $\tau^{-r}S[\infty]$
    are summands of $C$ with $2\leq r\leq p$, but the Pr\"ufer sheaves
    $\tau^{-}S[\infty],\dots,\tau^{-{(r-1)}}S[\infty]$ in between are
    not, then there is a (unique) connected branch in the wing $\Ww$
    rooted in $S[r-1]$ that occurs as a summand of $C$. 
  \item The torsionfree part $C_+$ of $C$ is $x$-divisible; thus
      $C$ is automatically of slope $\infty$ in this case.
    \end{enumerate}
  \end{enumerate}
\end{lemma}
\begin{proof}
  We assume without loss of generality that $C$ is minimal cotilting
  and hence discrete; thus $C$ is unqiuely determined by its
  indecomposable direct summands.\medskip
  
  Suppose that $C_x$ does not contain a Pr\"ufer sheaf as a
  direct summand. Then, by Theorem \ref{thm:torsion-splitting}, we
  must have that $C_x$ is a direct sum of exceptional sheaves in
  $\Uu_x$. There are only finitely many exceptional sheaves in $\Uu_x$
  and so let $B_1,\dots, B_s$ denote the exceptional summands of $C$
  (up to isomorphism).  Since $\bigoplus_{i=1}^s B_i$ is rigid, it has
  at most the number of summands of a tilting object in $\Uu_x$, such
  objects are well-known and it follows that $0 \leq s < p$,
  cf.~\cite[Ch.~3]{lenzing:meltzer:1996}.  By Lemma
  \ref{lem:maximality-in-tubes}, we must have that every adic
  $S_x[-\infty]$ such that $\Hom_{\QHh}(S_x[-\infty], \tau B_i) = 0$
  for all $1 \leq i \leq s$ is a direct summand of $C$; for future
  reference, let $\Aa$ denote the set of these adics.  This is a
  consequence of generalised Serre duality and the fact that
  $\Hom_{\QHh}(\tau^-B_i, S_x[-\infty]) = 0$ for all
  $1 \leq i \leq s$.  Since $s < p$, there is at least one coray that
  does not contain any $B_i$, therefore $\Aa$ is non-empty.  Now, by
  another application of Lemma \ref{lem:maximality-in-tubes}, we have
  that the $B_i$ are maximal self-orthogonal (and hence maximal rigid)
  among the sheaves in
  $\{ U \in \Uu_x \mid \tau U \in \Aa^{\perp_0}\}$.  It follows from a
  standard argument (e.g.~\cite[Ch.~3]{lenzing:meltzer:1996}) that
  $B_1, \dots B_s$ must form a branch sheaf.  It follows from the form
  of the branches that there are $p-s$ adics in $\Aa$.

  Assume the interior case. By
  Lemma~\ref{lem:pruefer-adic-ext-relations} no adic sheaf associated
  with the tube $\Uu_x$ can be a direct summand of $C$. Moreover, the
  same proof as in~\cite[Lem.~4.10]{angeleri:kussin:2017} shows that
  $C_+$ is $x$-divisible. By
  Lemma~\ref{lem:V-divisibles-have-slope-infty} thus $C$ has slope
  $\infty$. The proof concerning the claim for the wings in the
  interior case is completely dual to the arguments given in the proof
  of~\cite[Lem.~4.9]{angeleri:kussin:2017}, and we therefore omit it.
\end{proof}
Suppose $C$ is a cotilting sheaf of slope $\infty$ such that $C$ falls
into case (2) of Lemma~\ref{lem:exceptional-tube} with respect to
$x \in \XX$. It follows immediately from
Corollary~\ref{cor:properties-cotilting} that the branch summand $B$
of $C$ in $\Uu_x$, viewed as collection of indecomposable sheaves, is
given as $$B=\Prod(C)\cap\Uu_x.$$ In particular, this shows that a
cotilting sheaf $C'$ with a different branch $B'\neq B$ in $\Uu_x$
will have $\lperpe{C'}\neq\lperpe{C}$, that is, $C$ and $C'$ cannot be
equivalent.
\subsection*{The generating torsionfree classes}
We now consider a pair $(B,V)$ given by a branch sheaf $B\in\Hh$ and a
subset $V\subseteq\XX$, and we associate a generating torsionfree
class in $\Hh$ to it. In order to do this we next define two pieces of
notation.

Firstly, let $\Ww$ be the collection of pairwise non-adjacent wings in
$\Uu_x$ determined by the branch sheaf $B$. Let
$S, \tau^-S, \dots, \tau^{-(r-1)}S$ be a basis for one of the wings in
$\Ww$.  Then $\Rr_x$ is the set of indices $j \in\{0,\dots,p(x)-1\}$
such that $\tau^{j+1}S\notin \Ww$.

Secondly, given a connected branch $A$ with associated wing $\Ww_A$,
we define the \emph{undercut} of $A$ as
in~\cite[(4.9)]{angeleri:kussin:2017}:
\begin{equation*}
  A^{>}:=
\begin{cases}
  \rperpo{A}\cap\Ww_A & \text{if}\ A\ \text{is interior},\\
  \rperpo{A}\cap\tau\Ww_A & \text{if}\ A\ \text{is exterior}.
\end{cases}
\end{equation*}
The torsionfree class $\Ff_0$ associated to $(B,V)$ will consist of
all vector bundles, of the rays given by the sets $\Rr_x$, and of some
objects determined by $B$.  Up to $\tau$-shift, these objects will
belong to the wings defined by the undercut of $B$.
\begin{lemma}\label{lemma:undercut}
  Let $V\subseteq\XX$ and $B=B_{\mathfrak i}\oplus B_{\mathfrak e}$ be
  a branch sheaf.
  \begin{enumerate}
  \item[(1)] The class
    \begin{equation}
      \label{eq:undercut-formula}
      \Ff_0=\add\Bigl(\,\vect\XX\cup\tau^- (B^{>})\cup\bigcup_{x\in
        V}\{\tau^j S_x[n]\mid j\in \Rr_x,\,n\in\NN\}\,\Bigr)
    \end{equation}
    is a torsionfree class in $\Hh$ which
    generates. \smallskip
  \item[(2)] There is a cotilting sheaf $C$ with cotilting class
    $\lperpe{C}=\varinjlim\Ff_0$. For any such $C$ its torsion part is
    (up to multiplicities) given by
    $$C_0=B\oplus\bigoplus_{x\in V}\bigoplus_{j\in \Rr_x}\tau^j
    S_x[\infty].$$

  \item[(3)] If, moreover, $C$ is assumed to be minimal cotilting,
    then the indecomposable summands of its torsionfree part $C_+$ are
    given by the adic sheaves $S_y[-\infty]$ with $y\in\XX\setminus V$
    and $S_y\in\Uu_y$ simple such that
    $\Hom(S_y[-\infty],\tau B_{\mathfrak e})=0$, and in case
    $V=\emptyset$, additionally by the sheaf $\Kk$ of rational
    functions.
  \end{enumerate}
\end{lemma}
\begin{proof}
  (1) It is shown in~\cite[Lem.~4.11]{angeleri:kussin:2017} that
  $\Ff_0$ is a resolving class, which in our setting means, that
  $\Ff_0$ generates and is closed under subobjects and extensions.

  (2) By Theorem~\ref{thm:cotilting-is-of-finite-type} there is a
  cotilting object $C$ with $\lperpe{C}=\vec{\Ff_0}$. Given a simple
  object $S\in\Uu_x$, it follows from
  Corollary~\ref{cor:pruefer-summands-of-products} and the fact that
  $\vec{\Ff_0}$ is is a torsion-free class that we have that
  $S[\infty]$ is a direct summand of $C$ if and only if
  $\vec{\Ff_0} = \lperpe{C}$ contains the ray $\{S[n]\mid\,n\geq 1\}$.
  Therefore the objects $\tau^j S_x[\infty]$ with $x\in V$ and
  $j\in\Rr_x$ are precisely the Pr\"ufer summands of $C$. Moreover, by
  Lemma \ref{lem:properties-cotilting}, we have that
  $$\Ff_0\cap\rperpe{{\Ff_0}}=\Prod(C)\cap\Hh.$$ It remains to show that this class
  coincides with $\add(B)$, which then shows that the torsion part
  $C_0$ is as indicated.  Since, by Proposition
  \ref{prop:tf-gen-equals-resolve}, the resolving classes in $\Hh$
  coincide with the generating torsion-free classes, we may proceed in
  exactly the same way as in the proof of
  \cite[Lem.~4.11]{angeleri:kussin:2017}, replacing $\mathscr{S}$ with
  $\Ff_0$.
  
  (3) After Lemma~\ref{lem:exceptional-tube}~(1) it only remains to
  show that $C_+$ does not have another indecomposable summand, except
  $\Kk$, if $C$ is assumed to be minimal, since $C$ is uniquely
  determined by its indecomposable direct summands (or also by
  part~(2) of the preceding proposition). By
  Theorem~\ref{thm:ind-pure-inj-slope-infty} this could only be either
  the generic $\Kk$ or another adic. If $V\neq\emptyset$ then the
  generic is already in $\Prod(C)$ (since $C$ contains a Pr\"ufer
  summand), in case $V=\emptyset$ we have to add it. Additional adics
  are not possible, using Lemma~\ref{lem:pruefer-adic-ext-relations}.
\end{proof}
As a consequence we obtain
\begin{proposition}\label{prop:T-C-same-torsion}
  Let $C$ be a cotilting sheaf of slope $\infty$.
  \begin{enumerate}
  \item[(1)] Let $T$ be a corresponding tilting sheaf such that
    $\Gamma(T)=C$. Then the torsion parts $C_0$ and $T_0$ coincide up
    to ``multiplicities'': $\Add(C_0)=\Add(T_0)$.
  \item[(2)] Up to equivalence, $C$ is uniquely determined by its
    torsion part $C_0$. \qed
\end{enumerate}
\end{proposition}
\subsection*{The classification}
Next we present the main result of this section, which states that,
for any pair $(V,B)$, where $V \subseteq\XX$ is non-empty and $B$ is a
branch sheaf, there is a uniquely determined cotilting sheaf $C$ of
slope $\infty$.  Moreover, every cotilting sheaf of slope $\infty$
arises in this way.  Each $x\in V$ dictates that $\Uu_x$ is of
Pr\"ufer type and the interior part of $B$ with respect to $x$
dictates which Pr\"ufers belonging to $\Uu_x$ occur.  Similarly, the
set $\XX\setminus V$ and the exterior part of $B$ with respect to $V$
control the tubes of adic type.
\begin{theorem}\label{thm:large-cotilting-sheaves-in-general}
  Let $\QHh=\Qcoh\XX$, where $\XX$ is a weighted noncommutative
  regular projective curve.
  \begin{enumerate}
  \item[(1)] Let $V\subseteq\XX$ and $B\in\Hh_0$ be a branch
    sheaf. There is a unique large cotilting sheaf
    $C=C_+\oplus C_0$ of slope $\infty$ up to equivalence such that
    \begin{equation}
      \label{eq:large-torsion-part}
      C_0=B\oplus\bigoplus_{x\in V}\bigoplus_{j\in \Rr_x}\tau^j
      S_x[\infty],
    \end{equation}
    where the sets $\Rr_x\subseteq\{0,\dots,p(x)-1)\}$ are non-empty
    and are uniquely determined by $B$.\smallskip
  \item[(2)] Every cotilting sheaf of slope $\infty$ is, up to
    equivalence, as in~(1) and $C$ is $V$-divisible. (If
    $V\neq\emptyset$ hence $C$ is automatically of slope
    $\infty$.)\smallskip
  \item[(3)] Assuming $C$ to be minimal, the indecomposable summands
    of the torsionfree part $C_+$ are the following:
    \begin{itemize}
    \item the adic sheaves $\tau^{\ell} S_y[-\infty]$ with
      $y\in\XX\setminus V$ and $\ell$ such that
      $\tau^{\ell}S_y \not\in\tau\Ww$ for any wing $\Ww$ associated
      with an exterior branch part of $B$; if $V\neq\emptyset$ then
      $C_+$ is the pure-injective envelope of these adic sheaves;
    \item if $V=\emptyset$, additionally the sheaf of rational
      functions $\Kk$.
    \end{itemize}
  \end{enumerate}
\end{theorem}
\begin{proof}
  The result is a consequence of Lemma~\ref{lemma:undercut} and the
  fact, that the classes~\eqref{eq:undercut-formula} are just the
  (generating) torsionfree subclasses in $\Hh$ containing $\vect\XX$,
  which follows from the corresponding results on resolving classes
  and tilting sheaves, cf.~\cite[Sec.~4]{angeleri:kussin:2017}.
\end{proof}
\begin{numb}
  Let $V\subseteq\XX$ and $B$ be a branch sheaf. The large cotilting
  sheaf from Theorem~\ref{thm:large-cotilting-sheaves-in-general} will
  be denoted by $C_{(B,V)}$. With the large tilting sheaf $T_{(B,V)}$
  of finite type from~\cite[(4.6)]{angeleri:kussin:2017} if
  $V\neq\emptyset$ and $T_{(B,\emptyset)}=\bL'\oplus B$ with $\bL'$ a
  Lukas tilting sheaf (cf.~\cite[Prop.~4.5]{angeleri:kussin:2017}) in
  $\rperp{B}$, we have $\Gamma(T_{(B,V)})=C_{(B,V)}$ by
  construction. It follows that the following holds true:\medskip

  \emph{$C_{(B,V)}$ is tilting $\quad\Leftrightarrow\quad$ $T_{(B,V)}$
    is cotilting $\quad\Leftrightarrow\quad$ $T_{(B,V)}$ is
    pure-injective $\quad\Leftrightarrow\quad$ $V=\XX$.}\medskip

  Indeed: because of Proposition~\ref{prop:T-C-same-torsion}~(1),
    saying that the pure-injective torsion parts $C_0$ and $T_0$
    agree, we only need to consider the torsionfree parts $C_+$ and
  $T_+$, respectively. Here, $C_+$ is always
  pure-injective. By~\cite[Thm.~4.8]{angeleri:kussin:2017} we have
  that $T_+$ is $V$-divisible, hence by
  Lemma~\ref{lem:V-divisibles-have-slope-infty} of slope $\infty$,
  unless $V=\emptyset$. It is sufficient to show that $T_+$ is
  pure-injective if and only if $V=\XX$. If $V=\XX$ then
  $T_+\in\Add(\Kk)$ is pure-injective (and up to multiplicities and
  summands isomorphic to $\Kk$, $T_{(B,\XX)}=C_{(B,\XX)}$). If
  $V=\emptyset$, then $T_+$ is not pure-injective: otherwise the Lukas
  sheaf $\bL$ would be pure-injective and therefore $\bL$ must have an
  adic summand $S[-\infty]$. Let $S'$ be simple from a different,
  homogeneous tube. It is easy to see that $S'[-\infty]$ cogenerates
  every vector bundle, and hence also $\bL$, since $\bL$ is
  $\vect(\XX)$-filtered by \cite[Thm.~4.4]{angeleri:kussin:2017}
  together with \cite[Cor.~2.15(2)]{saorin:stovicek:2011}. Then
  $S'[-\infty]$ also cogenerates $S[-\infty]$. But from
  Proposition~\ref{prop:summands-of-products} we deduce
  $\Hom(S[-\infty],S'[-\infty])=0$. Thus we can assume that
  $\emptyset\neq V\subsetneq\XX$. Then again $T_+$ is not
  pure-injective: It follows from~\cite[Sec.~5]{angeleri:kussin:2017}
  that we can assume that $\XX$ is non-weighted and (compare also the
  proof of Lemma~\ref{lem:reduced-pi-are-adics}) that $T_+$ becomes a
  projective generator in $\QHh/\Tt_{V}$, and
  $A_V:=\End(T_+)=\varinjlim\Hom_{\Hh}(L',L)$ (where $L$ is the
  structure sheaf and the direct limit runs over all sub-line bundles
  $L'$ so that $L/L'$ has support in $V$) is a noncommutative Dedekind
  domain (\cite[Cor.~3.15]{kussin:2016} is the special case
  $A_{\XX\setminus\{x\}}=R_x$), and it is PI since
  $A_V\subseteq k(\Hh)$, with the latter finite over its centre. It
  follows from~\cite[Thm.~4.2]{marubayashi:1972b}, or
  \cite[Thm.~1.6]{prest:1998} (alternatively also from
  Lemma~\ref{lem:reduced-pi-are-adics}), that in case $U$ is a
  pure-injective indecomposable summand of $T_+$, it must be the
  $P$-adic completion $(\widehat{A_V})_P$ of $A_V$ for some non-zero
  prime ideal $P$. On the other hand, $U$ is a summand of $A_V$. But
  in $A_V$ the partial sums of the first powers of a non-zero element
  in $P$ yields a Cauchy sequence which has no limit in $A_V$. Thus,
  $U$ is not pure-injective.\smallskip

  The discussion also shows:\smallskip

  \emph{$C_+$ is reduced if and only if $V\neq\emptyset,\,\XX$. In any
    case, the reduced part of $C$ (resp.\ of $C_+$) is (up to
    $\Prod$-equivalence) of the form $\widehat{T}$ (resp.\
      $\widehat{T_+}$)}.\smallskip

  Here, the \emph{completion} $\widehat{M}$ of a sheaf $M$ is defined
  as $\prod_{x\in\XX}\widehat{M}^x$, with the $x$-\emph{completion}
  $\widehat{M}^x$ of $M$ defined as
  $\varprojlim_{M/M'\in\Uu_x}M/M'$, like in~\cite{vandenbergh:2001},
  \cite[2.4]{ringel:2004}. Note that $\widehat{M}^x=0$ if $M$ is
    $x$-divisible, and $\widehat{M}^x=M$ if $M\in\Uu_x$. Since
  $\rperp{{\Uu_y}}$ is closed under limits, $\widehat{M}^x$ is
  $y$-divisible for any $y\neq x$, and thus has slope $\infty$ by
  Lemma~\ref{lem:V-divisibles-have-slope-infty}. In $\Mod(R_{x})$ we
  have that $\widehat{M}^x$ is pure-injective since it is complete
  (cf.\ \cite[Thm.~11.4]{krylov:tuganbaev:2008}), and hence it is
  pure-injective also in $\QHh$. It follows then that $\widehat{M}^x$
  lies in $\Prod(\Uu_x)$ and is discrete. We obtain that $\widehat{M}$
  coincides with $\PE(\bigoplus_{x\in\XX}\widehat{M}^x)$, since they
  have the same indecomposable (pure-injective) summands by using
  Theorem~\ref{thm:ind-pure-inj-slope-infty} and
  Proposition~\ref{prop:summands-of-products}. \qed
\end{numb}
\subsection*{Maximal rigid objects in a large tube}
Following~\cite{baur:buan:marsh:2014}, we call an object $U$ in
  the direct limit closure $\vec{\Uu}$ of a tube $\Uu$ \emph{maximal
    rigid} if it is rigid and every indecomposable $Y\in\vec{\Uu}$
  satisfying $\Ext^1(U\oplus Y,U\oplus Y)=0$ is a direct summand of
  $U$. With the preceding results we
  complement~\cite[Cor.~4.19]{angeleri:kussin:2017} by the statements
  (2') and (3'):
\begin{corollary}\label{BBM}
  The following statements are equivalent for an object
  $U\in\vec{\Uu}$.
  \begin{enumerate}
  \item[(1)] $U$ is maximal rigid in $\vec{\Uu}$.
  \item[(2)] $U$ is tilting in $\vec{\Uu}$.
  \item[(2')] $U$ is cotilting in $\vec{\Uu}$.
  \item[(3)] $U$ is of Pr\"ufer type and it coincides, up to
    multiplicities, with the summand $(tT)_x$ supported at $x$ in the
    torsion part of some large tilting sheaf $T\in\QHh$.
  \item[(3')] $U$ is of Pr\"ufer type and it coincides, up to
    multiplicities, with the summand $(tC)_x$ supported at $x$ in the
   torsion part of some large cotilting sheaf $C\in\QHh$. \qed
  \end{enumerate}
\end{corollary}

\section{The case of positive Euler characteristic}

We assume that $\XX$ is of domestic type, that is, the normalized
\emph{orbifold Euler characteristic} $\chi'_{orb}(\XX)$ is
\emph{positive}. Let $\delta(\om)$ be the (negative) integer such that
for the slopes $\mu(\tau E)=\mu(E)+\delta(\om)$ holds for each
indecomposable vector bundle $E$. The collection $\Ee$ of
indecomposable vector bundles $F$ such that $0\leq\mu(F)<-\delta(\om)$
forms a slice in the sense of~\cite[4.2]{ringel:1984}, and
$T_{\her}:=\bigoplus_{F\in\Ff}F$ is a tilting bundle having a
finite-dimensional tame hereditary $k$-algebra $H$ as endomorphism
ring. We refer to~\cite[Prop.~6.5]{lenzing:reiten:2006}. In particular
$\bDerived{\QHh}=\bDerived{\Mod(H)}$, and this is also the repetitive
category of $\Mod(H)$. Denote by $\bp$ and $\bq$ the preprojective and
the preinjective component of $H$, respectively. Since $H$ is
hereditary, the tilting torsion pair $(\Tt,\Ff)$ in $\QHh$ induced by
$T_{\her}$ splits. Moreover, in $\Mod(H)$ there is the (split) torsion
pair $(\Qq,\Cc)$ with $\Qq=\Gen\bq$.
\begin{theorem}\label{thm:domestic-main-result}
    Let $\XX$ be a domestic curve.
    \begin{enumerate}
    \item[(1)] Each indecomposable pure-injective sheaf in $\Qcoh\XX$
      is either a vector bundle or has slope $\infty$.
    \item[(2)] Each large cotilting sheaf in $\Qcoh\XX$ has slope
      $\infty$.
    \end{enumerate}
  \end{theorem}
  Hence the classifications of indecomposable pure-injective sheaves
  and of large cotilting sheaves in the domestic case are given by
  Theorem~\ref{thm:ind-pure-inj-slope-infty} (plus the indecomposable
  vector bundles) and
  Theorem~\ref{thm:large-cotilting-sheaves-in-general}, respectively. 
  \begin{proof}
    (1) It is sufficient to proof the following: if $E$ is
    indecomposable pure-injective and there is a non-zero morphism to
    a vector bundle then $E$ is a vector bundle. The analogue in the
    module case is well-known
    (cf.~\cite[3.~Lem.~1]{crawley-boevey:1998}). Since $E$ is
    indecomposable, either $E\in\Tt$ or $E\in\Ff$. We regard $E$ as an
    object in $\bDerived{\QHh}=\bDerived{\Mod(H)}$. Thus either
    $E\in\Mod(H)$ or $E\in\Mod(H)[-1]$. We invoke
    Lemma~\ref{lem:pure-in-der-cat} applied to $\Derived{\QHh}$. Let
    first $E\in\Mod(H)$. Since there is even an epimorphism from $E$
    to a vector bundle $F$, thus $F\in\Tt$ and hence $F\in\bp$, the
    claim $E\in\bp\subseteq\vect\XX$ follows from the result in the
    module case. In case $E\in\Mod(H)[-1]$, we have
    $E\in\Qq[-1]$. Then $E[1]$ is a direct summand of a product of
    finite dimensional $H$-modules. Since $(\Qq,\Cc)$ is a torsion
    pair and $\Cc$ closed under products, we obtain
    $E[1]\in\Prod(\bq)$. Thus, by~\cite[2.2]{ringel:2004}, either
    $E[1]\in\Cc$ or $\Hom(\bq,E[1])\neq 0$. The first case is not
    possible since $E\in\QHh$. In the latter case we get
    $E\in\bq[-1]\subseteq\vect\XX$
    from~\cite[3.~Lem.~1]{crawley-boevey:1998}.

    (2) We make use of
    Proposition~\ref{prop:tilting-cotilting-invariances}~(4) and of
    the fact that the corresponding result for tilting sheaves is
    known, cf.\ \cite[Sec.~6]{angeleri:kussin:2017}.
  \end{proof}

\section{The case of Euler characteristic zero}\label{sec:Euler-zero}
\emph{Throughout this section let $\XX$ be a weighted noncommutative
  projective curve of orbifold Euler characteristic zero, and
  $\QHh=\Qcoh\XX$.}\medskip 

For general information on the tubular case we refer
to~\cite{lenzing:meltzer:1992}, \cite{lenzing:1997},
\cite[Ch.~13]{reiten:ringel:2006}, \cite[Ch.~8]{kussin:2009}
and~\cite[Sec.~13]{kussin:2016}, on the elliptic case to
\cite[Sec.~9]{kussin:2016}.

Let $\ovp$ denote the least common multiple of the weights
$p_1,\dots,p_t$, that is, $\ovp=1$ in the case where $\XX$ is elliptic
and let $\ovp>1$ if $\XX$ is tubular.  The formulae for the degree and
for the slope of a non-zero object $E\in\Hh$ simplify to
$\mu(E)=\frac{\deg(E)}{\rk(E)}\in\widehat{\QQ}=\QQ\cup\{\infty\}$,
with $\deg(E)=\frac{1}{\kappa\varepsilon}\DLF{L}{E}$,
cf.~\eqref{eq:deg-define}.

Since we are assuming $\chi'_{orb}(\XX)=0$, it follows that every
indecomposable coherent sheaf is semistable.  We therefore have the
following result, which is analogous to Atiyah's
classification~\cite{atiyah:1957}.
\begin{theorem}[{\cite[Prop.~8.1.6]{kussin:2009},
    \cite[Thm.~9.7]{kussin:2016}}]\label{thm:coh-X-alpha}
  For every $\alpha\in\widehat{\QQ}$ the full subcategory $\bt_\alpha$
  of $\Hh$ formed by the semistable sheaves of slope $\alpha$ is a
  non-trivial abelian uniserial category whose connected components
  form stable tubes; the tubular family $\bt_\alpha$ is parametrized
  again by a weighted noncommutative regular projective curve
  $\XX_\alpha$ over $k$ which satisfies $\chi'_{orb}(\XX_{\alpha})=0$
  and is derived-equivalent to $\XX$. \qed
\end{theorem}
We therefore
write $$\Hh=\bigvee_{\alpha\in\widehat{\QQ}}\bt_{\alpha}.$$ Note that
the component $\bt_{\infty}$ coincides with the sheaves of finite
length.
\subsection*{Quasicoherent sheaves with real slope}
We recall that by~\cite{reiten:ringel:2006,angeleri:kussin:2017} the
notion of slope is extended to all quasicoherent sheaves using the
following partitions of $\QHh$. For
$w\in\widehat{\RR}=\RR\cup\{\infty\}$ we define
$$\bp_w=\bigcup_{\alpha<w}\bt_\alpha\quad\quad\bq_w=\bigcup_{w<\beta}\bt_\beta,$$
where $\alpha,\,\beta\in\widehat{\QQ}$. We then have a partition of
$\Hh$ as $\Hh=\bp_w \vee\bt_w\vee\bq_w$ if $w$ is rational, and
$\Hh=\bp_w\vee\bq_w$ if $w$ is irrational.  We define the
\emph{sheaves of slope $w$} to be those contained in
$\Mm(w)=\Bb_w\cap\Cc_w$
where
$$\Cc_w=\rperpo{{\bq_w}}=\lperpe{\bq_w}\quad \text{and}
\quad\Bb_w=\lperpo{\bp_w}= \rperpe{{\bp_w}}.$$ For coherent sheaves
this definition of slope is equivalent to the former one given as
fraction of degree and rank, and for irrational $w$ there are only
non-coherent sheaves in $\Mm(w)$.  

The following fundamental statement can be found
in~\cite[Thm.~13.1]{reiten:ringel:2006},
\cite[Thm.~7.6]{angeleri:kussin:2017}.
\begin{theorem}[Reiten-Ringel]\label{thm:reiten-ringel}
  \begin{enumerate}
  \item[(1)] $\Hom(\Mm(w'),\Mm(w))=0$ for $w<w'$.
  \item[(2)] Every indecomposable sheaf has a well-defined slope
    $w\in\widehat{\RR}$. \qed
  \end{enumerate}
\end{theorem}
\subsection*{Cotilting sheaves that have a slope}
\begin{theorem}\label{thm:exist-Ww}
  Let $w\in\widehat{\RR}$.
  \begin{enumerate}
  \item[(1)] There is a large cotilting sheaf $\bW_w$ of
    slope $w$ and with cotilting class $\Cc_w$.
  \item[(2)] If $w$ is irrational, then $\bW_w$ is, up to equivalence,
    the unique cotilting sheaf of slope $w$.
  \end{enumerate}
\end{theorem}
\begin{proof}
  (1) The class $\add(\bp_w)\subseteq\Hh$ is torsionfree and generates
  $\Hh$. Thus there is, by
  Theorem~\ref{thm:cotilting-is-of-finite-type}, a cotilting sheaf
  $\bW_w$ with $\Cogen(\bW_w)=\varinjlim(\bp_w)=\Cc_w$. Moreover, by
  the cotilting property clearly $\bW_w\in\rperpe{{\Cc_w}}$, which is
  a subclass of $\Bb_w$. Since in $\Hh$ no cotilting object has a
  slope (cf.~\cite[Rem.~7.7]{angeleri:kussin:2017}), $\bW_w$ is large.
  
  (2) Let $C$ be cotilting of irrational slope $w$. Since $C\in\Bb_w$,
  we have $\bp_w\subseteq\lperpe{\bW_w}\subseteq\lperpe{C}$, and thus
  $\bp_w\subseteq\lperpe{C}\cap\Hh$. Since $\lperpe{C}\subseteq\Cc_w$,
  and since $w$ is irrational, we obtain
  $\lperpe{C}\cap\Hh\subseteq\bp_w$. We obtain
  $\lperpe{C}\cap\Hh=\bp_w=\lperpe{\bW_w}\cap\Hh$. From
  Theorem~\ref{thm:cotilting-is-of-finite-type} the result follows.
\end{proof}
\begin{example}
  For $w=\infty$ the Reiten-Ringel sheaf $\bW:=\bW_{\infty}$, given as
  the direct sum of the generic $\Kk$ and all the Pr\"ufer sheaves, is
  cotilting with class $\lperpe{\bW}=\Cc_\infty=\QHh$. Since $\Kk$
  belongs to $\Prod(S[\infty])$ for any Pr\"ufer sheaf, also
  $\bW/\Kk$ is cotilting and equivalent to $\bW$.
\end{example}
\subsection*{Interval categories}
Let $w\in\widehat{\RR}$. We denote by $\Hh\spitz{w}$ the full
subcategory of $\bDerived{\Hh}$ defined as
$$\bigvee_{\beta>w}\bt_{\beta}[-1]\vee
\bigvee_{\gamma\leq w}\bt_{\gamma}.$$ The category $\Hh\spitz{w}$ is a
($[-1]$-shifted) HRS-tilt of $\Hh$ in $\bDerived{\Hh}$ with respect to
the split torsion pair $(\Tt_{w},\Ff_{w})$ in $\Hh$ given by
$\Tt_{w}=\bigvee_{\beta>w}\bt_{\beta}$ and
$\Ff_{w}=\bigvee_{\gamma\leq w}\bt_{\gamma}$ and hence is abelian, see
\cite[I.~Thm.~3.3]{happel:reiten:smaloe:1996}
and~\cite[Prop.~2.2]{lenzing:skowronski:1996}. Moreover
$\bDerived{\Hh\spitz{w}}=\bDerived{\Hh}$, and $\Hh\spitz{w}$ is
hereditary abelian, satisfying Serre duality. If
$w=\alpha\in\widehat{\QQ}$, then by~\cite[Prop.~8.1.6]{kussin:2009},
\cite[Thm.~9.7]{kussin:2016} we have that
$\Hh\spitz{\alpha}=\coh\XX_\alpha$ for some curve $\XX_\alpha$ with
$\chi'_{orb}(\XX_{\alpha})=0$. If $k$ is algebraically closed, then
$\XX_\alpha$ is always isomorphic to $\XX$. The rank function on
$\Hh\spitz{\alpha}$ defines a linear form
$\rk_\alpha\colon\Knull(\Hh)\ra\ZZ$.
\begin{lemma}[{Reiten-Ringel~\cite{reiten:ringel:2006}}]\label{lem:reiten-ringel}
  For every $w\in\widehat{\RR}$ the pair $(\Gen(\bq_w),\Cc_w)$ is a
  torsion pair. If $w\in\widehat{\QQ}$, then the torsion pair
  splits. \qed
\end{lemma}
Let $\alpha\in\widehat{\QQ}$. By $\QHh\spitz{\alpha}=\Gg_{\alpha}[-1]$
we denote the $[-1]$-shifted heart of the t-structure in
$\Derived{\QHh}$ induced by the torsion pair
$(\lperpo{\bW_\alpha},\lperpe{\bW_\alpha})
=(\Gen\bq_{\alpha},\Cc_{\alpha})$. We have
$\QHh\spitz{\alpha}=\Qcoh\XX_\alpha$, cf.\
Theorem~\ref{thm:coh-X-alpha}. If $X\in\QHh$ has a rational slope
$\alpha$, then clearly $X\in\QHh\cap\QHh\spitz{\alpha}$ where the
intersection is formed in
$\Derived{\QHh}=\Derived{\QHh\spitz{\alpha}}$; in $\QHh\spitz{\alpha}$
then $X$ has slope $\infty$. In particular, $\QHh\spitz{\alpha}$ is
locally noetherian with
$\fp(\QHh\spitz{\alpha})=\Hh\spitz{\alpha}$.\medskip

Let $w$ be irrational. Similarly, we denote by
$\QHh\spitz{w}=\Gg_w[-1]$ the $[-1]$-shifted heart of the t-structure
in $\Derived{\QHh}$ associated with the cotilting torsion pair
$(\lperpo{\bW_w},\lperpe{\bW_w})=(\Qq_w=\lperpo{\Cc_w},\Cc_w)$. We
note that $\Qq_w=\Gen\bq_w$, which follows, arguing in $\QHh$, with
the same arguments as in~\cite[Lem.~1.3+1.4]{reiten:ringel:2006}
replacing ``finite length'' by ``noetherian''. In $\Hh$ this induces
the (splitting) torsion pair $(\Qq_w\cap\Hh,\Cc_w\cap\Hh)$, and we
have
$\Qq_w\cap\Hh=\Gen(\bq_w)\cap\Hh=\add\bq_w=\bigvee_{\beta>w}\bt_{\beta}$
and
$\Cc_w\cap\Hh=\Cogen\bW_w\cap\Hh=\bigvee_{\gamma<w}\bt_{\gamma}$. Since
$\QHh$ and $\Hh$ are hereditary, $(\Cc_w,\Qq_w[-1])$ and
$(\Cc_w\cap\Hh,(\Qq_w\cap\Hh)[-1])$ are splitting torsion pairs in,
respectively, $\QHh\spitz{w}$ and $\Hh\spitz{w}$. Moreover,
$\Cc_w=\varinjlim(\bp_w)$ and
$\Qq_w=\varinjlim(\bq_w)\subseteq\Bb_w$. The situation is illustrated
in Figure~\ref{fig:QHh<w>}.\medskip

Let $w\in\widehat{\RR}$. A sequence
$\eta\colon0\ra E'\stackrel{u}\lra E\stackrel{v}\lra E''\ra 0$ with
objects $E',E,E''$ in $\QHh\cap\QHh\spitz{w}$ is exact in $\QHh$ if
and only if it is exact in $\QHh\spitz{w}$; indeed, both conditions
are equivalent to
$E'\stackrel{u}\lra E\stackrel{v}\lra E''\stackrel{\eta}\lra E'[1]$
being a triangle in
$\Derived{\QHh}$. (Cf.~\cite[Prop.~2.2]{stovicek:kerner:trlifaj:2011}.)
\begin{figure}[htb]
  \begin{center}
    \definecolor{dark}{gray}{.85}
    \setlength{\unitlength}{0.75cm}
    \begin{picture}(14,4)
      \newsavebox{\Hw}
      \savebox{\Hw}(4.12,2.75)[c]{$\QHh\spitz{w}$}
      
      \put(2.61,0.62){\colorbox{dark}{\usebox{\Hw}}}
      
      \put(7.0,0.5){\line(0,1){3}}
      \put(6.9,0.2){\tiny{$w$}}
      
      \put(2.5,0.5){\line(0,1){3}}
      \put(1.9,0.2){\tiny{$w[-1]$}}
      \put(2.7,0.8){\tiny{$\Qq_w[-1]$}}
      
      \put(-0.3,0.5){\line(1,0){4.3}}
      \put(4,0.5){\line(0,1){3}}
      \put(-0.3,3.5){\line(1,0){4.3}}
      \put(1.2,4){$\QHh[-1]$}
      \put(3.15,3.6){\tiny{$\infty[-1]$}}
      
      \put(4.2,0.5){\line(1,0){4.3}}
      \put(8.5,0.5){\line(0,1){3}}
      \put(4.2,3.5){\line(1,0){4.3}}
      \put(5.9,4){$\QHh$}
      \put(8.2,3.6){\tiny{$\infty$}}
      \put(6.0,0.8){\tiny{$\Cc_w$}}
      
      \put(8.7,0.5){\line(1,0){4.3}}
      \put(13,0.5){\line(0,1){3}}
      \put(8.7,3.5){\line(1,0){4.3}}
      \put(10.2,4){$\QHh[1]$}
      \put(7.7,0.8){\tiny{$\Qq_w$}}
    \end{picture}
    \caption{The interval category $\QHh\spitz{w}$}
    \label{fig:QHh<w>}
  \end{center}
\end{figure}

For irrational $w$ we will show several properties of $\QHh\spitz{w}$
in Section~\ref{sec:irrational-slope}.
\subsection*{Injective Cogenerators}
\begin{proposition}
  Let $w\in\widehat{\RR}$. The cotilting sheaf $\bW_w$ yields an
  injective cogenerator of $\QHh\spitz{w}$.
\end{proposition}
\begin{proof}
  This follows from~\cite[Prop.~4.4]{coupek:stovicek:2019}.
\end{proof}
\begin{remark}
  Let $w\in\widehat{\RR}$. We can choose $\bW_w$ to be a
  \emph{minimal} injective cogenerator in $\QHh\spitz{w}$, and we will
  do so in the following. Then $\bW_w$ is discrete and its
  indecomposable summands correspond, up to isomorphism, bijectively
  to the simple objects in $\QHh\spitz{w}$.
\end{remark}
\begin{proposition}\label{prop:slope-w-objects}
  Let $w\in\RR\setminus\QQ$. The sheaves of slope $w$ are the pure
  subsheaves of products of copies of $\bW_w$.
\end{proposition}
\begin{proof}
  Let $E$ be a sheaf of slope $w$. Since $\bW_w$ is an injective
  cogenerator in $\QHh\spitz{w}$ there is a short exact sequence
  $0\ra E\ra C_0\ra C_1\ra 0$ in $\QHh\spitz{w}$ with
  $C_0\in\Prod(\bW_w)$. This sequence is pure-exact in $\QHh$: let
  $F\in\Hh$ be coherent, without loss of generality,
  indecomposable. We have to show that the sequence stays exact under
  $\Hom(F,-)$. Thus we can also assume that $\Hom(F,C_1)\neq 0$. Since
  $w$ is irrational, this means $F\in\bp_w$. Now $\Ext^1(F,E)=0$ since
  $E\in\Bb_w=\rperpe{{\bp_w}}$.
\end{proof}
\begin{corollary}\label{cor:pure-injectives-of-irrational-slope}
  Let $w\in\RR\setminus\QQ$. The class of pure-injective sheaves of
  slope $w$ is given by $\Prod(\bW_w)$. \qed
\end{corollary}
\subsection*{Rational slope}
Let $\alpha\in\widehat{\QQ}$ and $M\in\QHh$. Since the torsion pair
from Lemma~\ref{lem:reiten-ringel} splits, we have $M=M'\oplus M''$
with $M'\in\Cc_{\alpha}$ and $M''\in\Gen\bq_{\alpha}$.
\begin{lemma}
  Let $M\in\QHh$ be indecomposable of slope
  $\alpha\in\widehat{\QQ}$. Then $M$ is pure-injective in $\QHh$ if
  and only if $M$ is pure-injective considered as an object in
  $\QHh\spitz{\alpha}$.
\end{lemma}
\begin{proof}
  The class $\Mm(\alpha)=\Cc_\alpha\cap\Bb_\alpha$ is definable, both
  in $\QHh$ and in $\QHh\spitz{\alpha}$, in particular closed under
  forming products in $\QHh$ and $\QHh\spitz{\alpha}$. For every set
  $I$, forming the product $M^I$ in $\QHh$ is the same as forming the
  product $M^I$ in $\QHh\spitz{\alpha}$. Indeed, consider the product
  $M^I$ in $\QHh\spitz{\alpha}$ with projections $p_i\colon M^I\ra M$
  ($i\in I$). Let $X\in\QHh$ and $f\in\Hom(X,M)$. Write $X=X'\oplus
  X''$ as above, so that $X'\in\QHh\spitz{\alpha}$ and
  $\Hom(X'',M)=0$. By the universal property of the product $M^I$ in
  $\QHh\spitz{\alpha}$ there is a unique $\bar{f}\in\Hom(X',M^I)$ with
  $p_i \circ\bar{f}=f$ for all $i$. This property is then trivially
  extended to $X$, which shows, that $M^I$ with the projections $p_i$
  is also the product in $\QHh$. The converse direction is similar.

  The claim now follows with the Jensen-Lenzing criterion
  Proposition~\ref{prop:jensen-lenzing-crit}.
\end{proof}
\subsection*{Indecomposable pure-injective sheaves}
We obtain the following version for sheaves
of~\cite[Thm.~6.7]{angeleri:kussin:2017b}.
\begin{theorem}\label{thm:nondomestic-main-result}
  The following is a complete list of the indecomposable
  pure-injective sheaves in $\QHh=\Qcoh\XX$:
  \begin{enumerate}
  \item[(1)] The indecomposable coherent sheaves.
  \item[(2)] For every $\alpha\in\widehat{\QQ}$ the generic, the
    Pr\"ufer and the adic sheaves of slope $\alpha$.
  \item[(3)] For every irrational $w$ the indecomposable objects of
    $\Prod(\bW_w)$.
  \end{enumerate}
\end{theorem}
\begin{proof}
  We recall that each indecomposable object has a slope. Because of
  Corollary~\ref{cor:pure-injectives-of-irrational-slope} we only need
  to consider slopes $\alpha$ in $\widehat{\QQ}$, and by the preceding
  lemma we can restrict even to the case $\alpha=\infty$. Let now $M$
  be indecomposable of slope $\infty$. Then we can apply
  Theorem~\ref{thm:ind-pure-inj-slope-infty}. 
\end{proof}
\subsection*{Every large cotilting sheaf has a slope}
\begin{lemma}\label{lem:cotilt-tilt}
  Let $C$ be cotilting and $T$ be a corresponding tilting sheaf (of
  finite type): $C=\Gamma(T)$. Let $w\in\widehat{\RR}$. Then
  $$C\ \text{has slope}\ w\quad\Leftrightarrow\quad T\ \text{has
    slope}\ w.$$  
\end{lemma}
\begin{proof}
  We show the following:
  \begin{enumerate}
  \item[(1)] $C\in\Bb_w\quad\Leftrightarrow\quad T\in\Bb_w$.
  \item[(2)] $C\in\Cc_w\quad\Leftrightarrow\quad T\in\Cc_w$.
  \end{enumerate}
  To this end let $(\lperpo{C},\lperpe{C})$ and
  $(\rperpe{T},\rperpo{T})$ be the corresponding cotilting, resp.\
  tilting, torsion pairs. Moreover, let
  $\Ff=\lperpe{C}\cap\Hh=\lperpe{(\rperpe{T})}\cap\Hh=\vSs$ be the
  corresponding ``small'' torsionfree/resolving class. We have
  $\rperpe{T}=\rperpe{\vSs}$ and $\lperpo{C}=\lperpo{\vec{\Ff}}$. We
  remark that $\tau(\bp_w)=\bp_w$ and $\tau(\bq_w)=\bq_w$.

  (1) We have $C\in\Bb_w=\lperpo{\bp}_w$ iff
  $\bp_w\subseteq\rperpo{C}$ iff $\bp_w\subseteq\lperpe{C}$ iff
  $\bp_w\subseteq\Ff=\vSs$ iff $\bp_w\subseteq\rperpo{(\rperpe{T})}$
  iff $T\in\lperpo{\bp}_w=\Bb_w$.

  (2) We have $C\in\Cc_w=\rperpo{{\bq_w}}$ iff
  $\bq_w\subseteq\lperpo{C}=\lperpo{\vec{\Ff}}$ iff
  $\bq_w\subseteq\lperpo{\Ff}$ iff
  $\bq_w\subseteq\rperpe{\vSs}=\rperpe{T}$ iff
  $T\in\lperpe{\bq}_w=\rperpo{{\bq_w}}=\Cc_w$.
\end{proof}
The main result of this section is the following, which follows from
the lemma and the corresponding result for large tilting
sheaves~\cite[Thm.~8.5 + 9.1]{angeleri:kussin:2017}.
\begin{theorem}\label{thm:every-large-ts-slope}
  For every large cotilting sheaf $C$ in $\QHh$, there is
  $w\in\widehat{\RR}$ such that $C$ has slope $w$. \qed
\end{theorem}
\begin{example}
  For every $w\in\widehat{\RR}$ denote by $\bL_w$ the tilting sheaf in
  $\QHh$ with tilting class $\Bb_w$,
  cf.~\cite{angeleri:kussin:2017}. Then $\Gamma(\bL_w)=\bW_w$.
\end{example}
\subsection*{Reduction from rational slope to slope $\infty$}
\begin{lemma}\label{lem:tilting-corr-tubular}
  Let $\alpha\in\widehat{\QQ}$. For an object $C$ in $\QHh$ the
  following conditions are equivalent:
  \begin{enumerate}
  \item[(1)] $C$ is a cotilting sheaf in $\QHh$ of slope $\alpha$;
  \item[(2)] $C$ is a cotilting sheaf in $\QHh\spitz{\alpha}$ of slope
    $\infty$.
  \end{enumerate}
\end{lemma}
\begin{proof}
  Clearly, by changing the roles of $\QHh$ and $\QHh\spitz{\alpha}$,
  it suffices to show (1)$\Rightarrow$(2). Assuming~(1) we show (CS1),
  (CS2) w.r.t.\ $\QHh\spitz{\alpha}$ and that
  $\lperpe{C}\cap\Hh\spitz{\alpha}$, formed in $\QHh\spitz{\alpha}$,
  generates. For (CS1) it suffices to remark that forming the product
  $C^I$ in $\QHh$ and $\QHh\spitz{\alpha}$ yields the same; this
  follows from~\cite[Cor.~2.13]{coupek:stovicek:2019}. For (CS2) let
  $X\in\QHh\spitz{\alpha}$ such that
  $\Hom_{\QHh\spitz{\alpha}}(X,C)=0=\Ext^1_{\QHh\spitz{\alpha}}(X,C)$.
  Since the ``cut'' at $\bt_{\infty}[-1]$ defines
  by Lemma~\ref{lem:reiten-ringel} a splitting torsion pair
  $(\Tt_{\infty},\Ff_{\infty})$ in $\QHh\spitz{\alpha}$, we can write
  $X=X'\oplus X''$ with $X'\in\Tt_{\infty}$, that is, lying in $\QHh$,
  and $X''\in\Ff_{\infty}$, that is, lying in $\QHh[-1]$. Using (CS2)
  w.r.t.\ $\QHh$ (for $C$) and $\QHh[-1]$ (for $C[-1]$), we conclude
  $X'=0=X''$, and hence $X=0$. Moreover, the same splitting property
  shows that all objects from $\Ff_{\infty}$ belong to
  $\Ker\Ext^1_{\QHh\spitz{\alpha}}(-,C)$. This concludes the proof
  that $C$ is cotilting in $\QHh\spitz{\alpha}$.
\end{proof}
Let $B_\alpha$ be a sheaf of slope $\alpha$ that becomes a branch
sheaf of finite length in $\QHh\spitz{\alpha}$.  Then we call
$B_\alpha$ a \emph{branch sheaf of slope $\alpha$}.  Note that the
direct summands of $B_\alpha$ are contained in a subcategory
$\Ww_\alpha$ that becomes a wing in $\QHh\spitz{\alpha}$.  We call
$\Ww_\alpha$ a wing of slope $\alpha$ and we adopt all of the
appropriate notation and terminology suggested by
Section~\ref{nr:wings}.

We conclude this chapter by summarizing our results on large cotilting
sheaves in the tubular and the elliptic cases.
\begin{theorem}\label{thm:Euler-zero-cotilting}
  Every large cotilting sheaf $C$ (minimal, without loss of
  generality) in $\QHh$ has a slope $w\in\widehat{\RR}$, and for
  irrational $w$ we have $C\cong\bW_w$. Let $\alpha$ be rational or
  infinite.
  \begin{enumerate}
  \item[(1)] Let $V_\alpha\subseteq\XX_\alpha$ and $B_\alpha$ be a
    branch sheaf of slope $\alpha$. There is a unique minimal
    cotilting sheaf $C=C_+\oplus C_0$ of slope $\alpha$ whose torsion
    part is given by
    \begin{equation*}
      C_0=B_\alpha\oplus\bigoplus_{x\in V_\alpha}\bigoplus_{j\in \Rr_x}\tau^j
      S_x[\infty],
    \end{equation*}
    where the non-empty sets $\Rr_x\subseteq\{0,\dots,p(x)-1)\}$ are
    uniquely determined by $B_\alpha$ as in \ref{nr:wings}.\smallskip
  \item[(2)] Every cotilting sheaf of slope $\alpha$ is, up to
    equivalence, as in~(1).\smallskip
  \item[(3)] The indecomposable summands of the torsionfree part $C_+$
    of $C$ are the following:
    \begin{itemize}
    \item the adic sheaves $\tau^{\ell} S_y[-\infty]$ with
      $y\in\XX_\alpha\setminus V_\alpha$ and $\ell$ such that
      $\tau^{\ell}S_y \not\in\tau\Ww$ for any wing $\Ww$ associated
      with an exterior branch part of $B_\alpha$; if
      $V_\alpha\neq\emptyset$ then $C_+$ is the pure-injective
      envelope of these adic sheaves;
    \item if $V_\alpha=\emptyset$, additionally the generic sheaf of
      slope $\alpha$.\smallskip
    \end{itemize}
  \item[(4)] If $V_\alpha=\XX_\alpha$ and $\Rr_x=\{0,\dots,p(x)-1)\}$
    for all $x$, then $C\cong\bW_\alpha$.\qed
  \end{enumerate}
\end{theorem}

\section{Additional results related to irrational slopes}\label{sec:irrational-slope}
\emph{We continue to assume that the orbifold Euler characteristic
  $\chi'_{orb}(\XX)$ is zero. Throughout, we let $w$ be
  irrational.}\medskip

Our understanding of $\Prod(\bW_w)$, the class of pure-injectives in
$\QHh$ of slope $w$, is still quite small. The natural home of the
object $\bW_w$ is the category $\QHh\spitz{w}$, of which it is an
injective cogenerator. One should regard this Grothendieck category as
a geometrical object (in the sense of noncommutative algebraic
geometry, cf.\ the introductions in~\cite[1.2]{vandenbergh:2001} or
\cite[Ch.~III]{rosenberg:1995}), where the points are given by the
simple objects, or equivalently, by the indecomposable objects in
$\Prod(\bW_w)$. Some of the statements in the following proposition
were already stated in~\cite[Rem.~7.5]{angeleri:kussin:2017} without
proofs; part~(2) was obtained in discussions with H.~Lenzing.
\begin{proposition}\label{prop:irrational-cut}
  The following holds.
  \begin{enumerate}
  \item[(1)] $\QHh\spitz{w}$ is a locally coherent Grothendieck
    category with
    $\Hh\spitz{w}=\fp(\QHh{\spitz{w}})=\coh(\QHh{\spitz{w}})$.
  \item[(2)] $\Hh\spitz{w}$ does not contain any simple object.
  \item[(3)] Every non-zero object in $\Hh\spitz{w}$ is not
    noetherian, and thus $\QHh\spitz{w}$ is not locally noetherian.
  \item[(4)] There exist simple objects in $\QHh\spitz{w}$.
  \end{enumerate}
\end{proposition}
\begin{proof}
  (1) This follows from Theorem~\ref{thm:loc-coh-hearts}.
  
  (2) We assume that there is a simple object $S$ in
  $\Hh\spitz{w}$. Then there is a rational $\alpha<w$ such that
  $S\in\bt_\alpha$. We choose a rational $\beta$ with
  $\alpha<\beta<w$. The sheaf category $\Hh\spitz{\beta}$ defines a
  rank function $\rk_{\beta}$, which is additive on short exact
  sequences in particular in $\Hh\spitz{w}\cap\Hh\spitz{\alpha}$ and
  $\tau$-invariant. Moreover, $\rk_\beta(F)>0$ for every
  indecomposable $F$ in $\Hh\spitz{w}\cap\Hh\spitz{\alpha}$. We choose
  $F$ such that $\rk_\beta(F)$ is minimal, and moreover with
  $F\in\bt_\gamma$ such that $\gamma<\alpha$. By
  \cite[Thm.~13.8]{kussin:2016} we may assume that
  $\Hom_{\Hh\spitz{w}}(F,S)\neq 0$. Since $S$ is simple, there is a
  short exact sequence $0\ra U\ra F\ra S\ra 0$ in $\Hh\spitz{w}$, and
  by the choice of $F$ we get $\rk_{\beta}(U)=0$, that is,
  $U=0$. Hence we get an isomorphism $F\cong S$, which gives a
  contradiction since $F$ and $S$ have different slopes.
  
  (3) Since a non-zero noetherian object has a maximal subobject, this
  follows directly from (2).

  (4) Let $E$ be a non-zero, finitely generated object in
  $\QHh\spitz{w}$ (for instance, $E\neq 0$ finitely presented). Then
  it contains a maximal subobject, and the quotient is simple. (Thus
  one might expect that there are even ``many'' simple objects in
  $\QHh\spitz{w}$.)
\end{proof}
\begin{corollary}\label{cor:Ww-Sigma-pi}
  For $w\in\widehat{\RR}$, the cotilting sheaf $\bW_w$ is
  $\Sigma$-pure-injective if and only if $w\in\widehat{\QQ}$.
\end{corollary}
\begin{proof}
  This follows from~\cite[Prop.~V.4.3]{stenstroem:1975} and the fact
  that the category $\QHh\spitz{w}$ is locally noetherian if and only
  if $w\in\widehat{\QQ}$.
\end{proof}
\begin{proposition}\label{prop:simples-and injectives-have-slope-w}
  The class of injective objects in $\QHh\spitz{w}$ is given by
  $\Prod(\bW_w)$, where $\Prod$ can be formed either in
  $\QHh\spitz{w}$ or in $\QHh$. Each injective object and each simple
  object in $\QHh\spitz{w}$ has ``internal'' slope $w$, that is,
  belongs to $\lperpo{\Hh\spitz{w}}$, this class of objects formed in
  $\QHh\spitz{w}$.
\end{proposition}
\begin{proof}
  The statement on forming $\Prod$ follows
  from~\cite[Cor.~2.13]{coupek:stovicek:2019}. Every injective object
  $Q$ in $\QHh\spitz{w}$ is a direct summand of a power ${\bW_w}^I$ of
  $\bW_w$ (for some set $I$). Since $\lperpo{\Hh\spitz{w}}$ is closed
  under products, which follows by the same arguments as
  in~\cite[Prop.~13.5]{reiten:ringel:2006}, we conclude that
  $Q\in\lperpo{\Hh\spitz{w}}$.

  Let $S$ be a simple object in $\QHh\spitz{w}$. If
  $S\not\in\lperpo{\Hh\spitz{w}}$. Then there is a monomorphism
  $S\ra F$ for an object $F\in\Hh\spitz{w}$. Since $F$ is coherent and
  $S$ finitely generated, we obtain $S\in\Hh\spitz{w}$, and $S$ is
  simple in $\Hh\spitz{w}$. This yields a contradiction by
  Proposition~\ref{prop:irrational-cut}.
\end{proof}
\begin{remark}
  The statement in the preceding proposition on simple objects in
  $\QHh\spitz{w}$ is also shown in~\cite[Thm.~8.2.3]{rapa:2019} with
  completely different methods. Moreover, based on ideas by J.\
  \v{S}{\v{t}}ov{\'{\i}}{\v{c}}ek, in that thesis a simple object in
  $\QHh\spitz{w}$ is constructed in an explicit way as a direct limit
  of finitely presented objects of rational slopes.
\end{remark}
\begin{question}\label{q:w-hereditary}
  Is $\QHh\spitz{w}$ hereditary?
\end{question}
This interesting question is still open. We know that $\Hh\spitz{w}$
is hereditary, and by considering the derived category we see that at
least $\Ext^i_{\QHh\spitz{w}}(-,-)=0$ for $i\geq 3$. Moreover, if $S$
is a simple object (or any object of slope $w$) in $\QHh\spitz{w}$,
then $\IE(S)/S$ is injective. For heredity we would need that every
factor object of an injective object is injective. At least
$\QHh\spitz{w}$ is \emph{semihereditary}\footnote{For the
  corresponding ring-theoretic notion we refer
  to~\cite{megibben:1970}.} in the following sense:
\begin{proposition}
  In $\QHh\spitz{w}$ each of the following equivalent conditions holds
  true.
\begin{enumerate}
\item[(1)] Each factor object of an fp-injective object is fp-injective.
\item[(2)] Each factor object of an injective object is
    fp-injective.
\item[(3)] For all $X,\,Y\in\QHh\spitz{w}$, with $X$ finitely
  presented, $\Ext^2_{\QHh\spitz{w}}(X,Y)=0$.
\end{enumerate}
Moreover, the fp-injective objects coincide with the objects of slope
$w$ in $\QHh\spitz{w}$, and they form a definable subcategory of
$\QHh\spitz{w}$.
\end{proposition}
\begin{proof}
  The equivalence of the conditions follows from standard arguments by
  applying $\Hom_{\QHh\spitz{w}}(X,-)$ with $X$ finitely presented to
  a short exact sequence of the form $0\ra Y\ra Q\ra Q/Y\ra 0$ with
  $Q$ injective or fp-injective.

  In $\QHh\spitz{w}$ we have (generalised) Serre duality
  $\D\Ext^1_{\QHh\spitz{w}}(X,Y)=\Hom_{\QHh\spitz{w}}(Y,\tau X)$,
  where $X,\,Y\in\QHh\spitz{w}$ with $X$ finitely presented. Indeed,
  this follows from Lemma~\ref{lem:Serre}, applied to the derived
  category $\Derived{\QHh\spitz{w}}=\Derived{\QHh}$ and using that
  $\QHh$ is locally noetherian and the compact objects are given by
  $\bDerived{\Hh}$, cf.~\ref{numb:loc-noeth-comp-gen}. 

   By this we see
    $\lperpo{\Hh\spitz{w}}
    =\{Y\in\QHh\spitz{w}\mid\Hom_{\QHh\spitz{w}}(Y,\Hh\spitz{w})=0\}
    =\{Y\in\QHh\spitz{w}\mid\Ext^1_{\QHh\spitz{w}}(\Hh\spitz{w},Y)=0\}
    =\rperpe{\Hh\spitz{w}}$, and hence the objects of slope $w$
    coincide with the fp-injective objects in $\QHh\spitz{w}$. Since
    every factor object of an injective has slope $w$, by
    Proposition~\ref{prop:simples-and injectives-have-slope-w}, it is
    fp-injective. Thus~(2) holds.

    Moreover, it follows as in
    Proposition~\ref{prop:slope-infty-definable} that the class of
    objects in $\QHh\spitz{w}$ of slope $w$ is definable.
\end{proof}
Since $\QHh\spitz{w}$ is not locally noetherian, there are
fp-injective objects which are not injective
(\cite[Prop.~A.11]{krause:2001}).\medskip

We discuss several equivalent formulations of
Question~\ref{q:w-hereditary}.
\begin{lemma}[Reiten-Ringel]\label{lem:2nd-construction}
  Let $w$ be irrational, $\beta_1>\beta_2>\dots>w$ a sequence of
  rational numbers converging to $w$ and $Q_i\in\Add(\bt_{\beta_i})$
  (for $i=1,2,\dots$). Then in $\QHh$ we have
  $\prod Q_i/\bigoplus Q_i\in\Mm(w)$.
\end{lemma}
\begin{proof}
  This is a slightly more general version of ``The First
  Construction'' in~\cite[13.4]{reiten:ringel:2006}; the proof therein
  still holds.
\end{proof}
\begin{proposition}
  Let $w$ be irrational. The following are equivalent:
  \begin{enumerate}
  \item[(1)] The abelian category $\QHh\spitz{w}$ is
    hereditary.\smallskip 
  \item[(2)] The torsion pair $(\Qq_w,\Cc_w)$ in $\QHh$
    splits.\smallskip  
  \item[(3)] $\Ext^1_{\QHh}(\Prod(\bW_w),\Add(\bq_w))=0$; in the
    second argument, one can restrict to coproducts of objects in
    $\bq_w$ whose slopes converge to $w$.\smallskip 
  \item[(4)] For each sequence $Q_i\in\Add(\bt_{\beta_i})$ with
    $\beta_1>\beta_2>\dots>w$ converging to $w$ the canonical
    monomorphism $\bigoplus Q_i\ra\prod Q_i$ splits.\smallskip  
  \item[(5)] In the category $\QHh\spitz{w}$ the following holds: for
    all objects $X\in\QHh\spitz{w}$ of the form $X=\bigoplus X_i$ with
    $X_i\in\Add(\bt_{\gamma_i})$, the $\gamma_i$ converging to
    $w[-1]$, and for each monomorphism $f$ from $X$ to an injective
    object in $\QHh\spitz{w}$, and for any object $Y$ of slope $w$ we
    have $\Ext^1_{\QHh\spitz{w}}(Y,\Coker f)=0$ (or equivalently,
    $\Ext^2_{\QHh\spitz{w}}(Y,X)=0$).
  \end{enumerate}
\end{proposition}
\begin{proof}
  (1)$\Rightarrow$(2) Let $\QHh\spitz{w}$ be hereditary. Let
  $0\ra X'\ra X\ra X''\ra 0$ be a short exact sequence in $\QHh$ with
  $X'\in\Qq_w$ and $X''\in\Cc_w$. Then
  $V=X'[-1]\in\Qq_w[-1]\subseteq\QHh\spitz{w}$, and this yields
  $\Ext^1_{\QHh}(X'',X')=\Hom_{\Derived{\QHh}}(X'',X'[1])=\Ext^2_{\QHh\spitz{w}}(X'',V)=0$. Hence
  $(\Qq_w,\Cc_w)$ splits.

  (2)$\Rightarrow$(1) Conversely, we assume that $(\Qq_w,\Cc_w)$
  splits. Let $X,\,Y\in\QHh\spitz{w}$. Since $(\Cc_w,\Qq_w[-1])$ is a
  torsion pair in $\QHh\spitz{w}$, for showing
  $\Ext^2_{\QHh\spitz{w}}(X,Y)=0$ it is sufficient to assume
  $X,\,Y\in\Cc_w\cup\Qq_w[-1]$. Since
  $\Ext^2_{\QHh\spitz{w}}(X,Y)=\Hom_{\Derived{\QHh}}(X,Y[2])$ and
  $\QHh$ is hereditary, the only crucial case is when $X\in\Cc_w$ and
  $Y\in\Qq_w[-1]$. But then $Y[1]\in\Qq_w$ and
  $\Hom_{\Derived{\QHh}}(X,Y[2])=\Ext^1_{\QHh}(X,Y[1])=0$ since
  $(\Qq_w,\Cc_w)$ splits.

  (2)$\Leftrightarrow$(3) is easy to show since $\QHh$ is hereditary,
  and moreover
  $\Ext^1_{\QHh}(\Cc_w,\Bb_\beta)=\Ext^2_{\QHh\spitz{\beta}}(\Cc_w,\Bb_{\beta}[-1])=0$
  for every rational $\beta>w$, because $\QHh\spitz{\beta}$ is
  hereditary.

  (3)$\Rightarrow$(4) Follows directly with
  Lemma~\ref{lem:2nd-construction}.

  (4)$\Rightarrow$(3) Let
  $\eta\colon 0\ra\bigoplus_i Q_i\ra E\ra {\bW_w}^J\ra 0$ be a short
  exact sequence with $Q_i$ as in~(4). We have to show that $\eta$
  splits. Since each $Q_i\in\Add(\bt_{\beta_i})\subseteq\Bb_{\beta_i}$
  with $\beta_i>w$ we have $\Ext^1_{\QHh}({\bW_w}^J,Q_i)=0$ for each
  $i$, hence we obtain $\Ext^1_{\QHh}({\bW_w}^J,\prod Q_i)=0$
  by~\cite[Cor.~A.2]{coupek:stovicek:2019}. Since by~(4) the coproduct
  $\bigoplus Q_i$ is a direct summand of the product $\prod Q_i$, we
  obtain $\Ext^1_{\QHh}({\bW_w}^J,\bigoplus Q_i)=0$, and thus $\eta$
  splits.

  (1)$\Rightarrow$(5) This is clear.

  (5)$\Rightarrow$(1) We recall that there is the splitting torsion
  pair $(\Cc_w,\Qq_w[-1])$ in $\QHh\spitz{w}$, and since
  $\QHh\spitz{w}$ is locally coherent with
  $\fp(\QHh\spitz{w})=\Hh\spitz{w}$, the class $\Qq_w[-1]$ is
  generated in $\QHh\spitz{w}$ by $\bq_w[-1]$. We have to show that
  for any short exact sequence $0\ra X\ra{\bW_w}^I\ra B\ra 0$ in
  $\QHh\spitz{w}$, the cokernel object $B$ is injective, that is,
  $\Ext^2_{\QHh\spitz{w}}(Y,X)=0$ holds for each $Y\in\QHh\spitz{w}$.
  Since $\Ext^3_{\QHh\spitz{w}}(-,-)=0$, we obtain that for $X$ it is
  sufficient to assume that it is a coproduct of objects in
  $\bq_w[-1]$. Moreover, since for every rational $\beta$ we have
  $\Hom_{\Derived{\QHh\spitz{\beta}}}(M,N)=0$ for all
  $M\in\QHh\spitz{\beta}[-1]$ and $N\in\QHh\spitz{\beta}[1]$ (which
  follows from heredity of $\QHh\spitz{\beta}$), we deduce that we
  can, moreover, assume $X$ to be of the form as in~(5), and to test
  injectivity with objects $Y$ of slope $w$.
\end{proof}
The equivalence of (1) and (2) also follows
from~\cite[Thm.~5.2]{stovicek:kerner:trlifaj:2011}. 

\subsection*{Acknowledgements}
For several inspiring discussions on the subject we would like to
thank Lidia Angeleri H\"ugel, Helmut Lenzing, Alessandro Rapa, Jan
\v{S}\v{t}ov\'{\i}\v{c}ek and Jorge Vitòria.

The first named author was supported by the DFG through SFB/TR 45 at
the University of Bonn. The second named author was supported by the
Max Planck Institute for Mathematics and also by the European Union's
Horizon 2020 research and innovation programme under the Marie
Sk{\l}odowska-Curie Grant Agreement No.~797281.


\bibliographystyle{elsarticle-harv}

\def\cprime{$'$}

\end{document}